\documentclass[reqno,15pt]{amsart}
\usepackage{amsmath}
\usepackage{amssymb}
\usepackage{amsfonts}
\usepackage{amscd}
\usepackage{mathrsfs}
\usepackage{graphics}
\usepackage{latexsym}
\usepackage{xcolor}

\usepackage[hyperfootnotes=false]{hyperref}
\usepackage{quiver}
\usepackage{ytableau}
\usepackage{float}

\usepackage{comment}

\theoremstyle{plain}
\newtheorem{theorem}{Theorem}[section]

\newtheorem{cor}[theorem]{Corollary}
\newtheorem{lem}[theorem]{Lemma}
\newtheorem{prop}[theorem]{Proposition}

\newcounter{kludge}
\newcounter{kludgeb}
\theoremstyle{definition}
\newtheorem{defn}[theorem]{Definition}

\newtheorem{rmk}[theorem]{Remark}

\theoremstyle{remark}

\input xy
\xyoption{all}

\newcommand{\marpar}[1]{}

\newcommand{\mbb}{\mathbb}
\newcommand{\QQ}{\mbb{Q}}

\newcommand{\LL}{\mbb{L}}
\newcommand{\ZZ}{\mbb{Z}}
\newcommand{\CC}{\mbb{C}}

\newcommand{\AAA}{\mbb{A}}

\newcommand{\mc}{\mathcal}

\newcommand{\mcE}{\mc{E}}

\newcommand{\mcI}{\mc{I}}
\newcommand{\mcJ}{\mc{J}}

\newcommand{\mcL}{\mc{L}}

\newcommand{\mcO}{\mc{O}}

\newcommand{\mcQ}{\mc{Q}}

\newcommand{\mcU}{\mc{U}}

\newcommand{\mcX}{\mc{X}}

\newcommand{\bfv}{\mathbf{v}}
\newcommand{\bfw}{\mathbf{w}}

\newcommand{\Ext}{\mathrm{Ext}}

\newcommand{\Hilb}{\text{Hilb}}
\newcommand{\Quot}{\text{Quot}}

\newcommand{\Hom}{\mathrm{Hom}}
\newcommand{\rk}{\mathrm{rank}}
\newcommand{\Sym}{\mathrm{Sym}}

\newsavebox{\sembox}
\newlength{\semwidth}
\newlength{\boxwidth}

\newsavebox{\semrbox}
\newlength{\semrwidth}
\newlength{\boxrwidth}

\newcommand{\ra}{\rightarrow}

\title{Total Cohomology of Hilbert schemes of Orbisurfaces I: Reduction to Stacky Loci}
\author{Shuo Gao}
\date{8 October 2026}
\address{Department of Mathematics, Stony Brook University, Stony Brook, NY}
\email{shuo.gao@stonybrook.edu}

\providecommand{\Lef}{\mathbb L}

\newtheorem{corollary}[theorem]{Corollary}

\theoremstyle{definition}

\newtheorem{remark}[theorem]{Remark}
\newtheorem{example}[theorem]{Example}
\allowdisplaybreaks
\begin{document}
\begin{abstract}
We generalize Nakajima's theory of Hilbert schemes of points on surfaces to Hilbert schemes of points on \textit{orbisurfaces},
and obtain initial factorization based on stratification by stabilizers.
While the usual formula is retained for the trivial stabilizer stratum and for those strata of codimension $1$,
genuine difficulties arise with stacky points strata.
In the local picture of an isolated stacky point, the theory reduce to the study of $\Quot^\bfv[\CC^2/G]$ for a \textit{small} finite subgroup of $GL_2(\CC)$.
The generating series for their total cohomology is a geometric correction of the generating series for total cohomology of ordinary smooth surface schemes.
In addition, we obtain an action of the Borel-Moore homology of the punctual fibre stacks on the total Hilbert cohomology by virtual correspondences,
in analogy with the Fock space description in Nakajima's theory.
The punctual fibre stacks are genuinely different from the symmetric stack, and has natural realizations via motivic integration.
\end{abstract}
\maketitle

\section{Introduction}
\label{sec:intro}
Mathematicians has long realized the advantage of taking account of all moduli of a unifying theme, and for Hilbert schemes of point this work was pioneered by Nakajima in \cite{Nakajima_1999},
yielding the clean and rich Fock space description of the total cohomology $\bigoplus_n H^*(\Hilb^n \CC^2)$ of $\Hilb^n\CC^2$; the Fock space description explains the G\"ottsche type formula
    \[\sum_{n,d\geq 0} \dim H^d(\Hilb^n\CC^2,\QQ)q^dt^n=\prod_{m\geq1}\frac{1}{1-q^{2m-2}t^m}=:F(q,t).\]
    As Nakajima's theory has nice direct applications to all smooth surfaces, different generalizations or enlargement of the notion of smooth surfaces were considered,
with the best and most natural enlargement to smooth 2-dimensional orbifolds, or smooth \textbf{orbisurfaces} --- which are precisely connected,
smooth, separated Deligne-Mumford effective stacks of finite type and of pure dimension 2 over $\CC$;
    not only was it possible to define the moduli space of substacks of orbisurfaces with specified types (replacing Hilbert polynomials),
but also the said moduli space is indeed a scheme: for each type $\alpha$ given by a integer length and correctional representation type at each stacky locus as introduced in section  \ref{subsec:alpha-type}.
    \begin{theorem}[\cite{olsson2003quot}, restated as theorem \ref{thm:global-hilbert-smooth}]
Let $\mcX$ be an orbisurface over $\CC$ and fix its  type $\alpha$.
        Then $\Quot^\alpha\mcX$ is a smooth separated quasi-projective scheme locally of finite type.
        If in addition the coarse space is projective, these Quot schemes of relative dimension 0 are also projective.
    \end{theorem}

    This article attempts to provide a unifying theme to consider $\Quot^\alpha\mcX$ for all orbisurfaces $\mcX$,
with an emphasis on the \textbf{balanced} type $(n,\vec{0})$ with all vanishing correction terms;
    we slightly abuse our notation by denoting such type by the integer $n$ as well, and define the Hilbert scheme to be $\Hilb^n\mcX:=\Quot^{(m,\vec{0})}\mcX$.
    Proposition \ref{prop:finite-stacky-locus} gives a initial topological partition $\mcX = \Delta_0\coprod\Delta_1\coprod\mcU$ be a topological partition of $\mcX$ into stacky loci of dimension $0$,
    stacky loci\footnote{possibly with a finite set of points removed to be disjoint from $\Delta_0$.}
of dimension $1$ and the open locus with trivial stabilizer, all of finite connected components.
    
    Topologically, we can define
    \[B_\mcX(q,t) = \sum_{n,d\geq 0} \dim H^d(\Hilb^n\mcX,\QQ) q^dt^n,\]
    and similarly with support specified, we can define the punctual Hilbert scheme $\Hilb^n_{Z}\mcX \subset \Hilb^n\mcX$ to be the locus of balanced quotients supported in $Z$ for locally closed substack $Z\subset\mcX$ with punctual series
    \[B_{\mcX,Z}(t)=\sum_{n\geq0}[\Hilb^n_{Z}\mcX]t^n.\]
    The support-decomposition argument of Gusein-Zade-Luengo-Melle-Hern\'andez,
    \[B_{\mcX} = B_{\mcX,\Delta_0} \cdot B_{\mcX,\Delta_1}\cdot  B_{\mcX,\mcU}=B_{\mcX,\Delta_0} \cdot B_{\mcX,\Delta_1}\cdot F_\mcU \]
    along $\mcX=\Delta_0\coprod\Delta_1\coprod\mcU$, by \cite[Theorem 1]{GuseinZadeLuengoMelle2010};
    note $F_\mcU$ is the classical generating series restated in section \ref{subsec:surface_series}.
    In a local chart with stabilizer group $H$, the balanced type $n$ means the quotient has character $n\CC[H]$.
    This is the local balancing condition of type (3) equivariant Hilbert scheme in \textit{loc.cit.};
    in the surface case it is also identified with type (2),
    the principal component containing configurations of distinct free orbits.
    Thus, for a small $G\subset GL_2(\CC)$, our $\Hilb^n[\CC^2/G]$ is their ${}^{(2)}\Hilb_{\CC^2}^{G,n|G|}$.
    On a general stack, balancing condition is imposed at each support point.
    The motivic factorization reduces the additive series to punctual factors on the stabilizer strata.
    The ordinary Betti factorization requires the support decomposition of Section \ref{sec:support},
    since ordinary Betti polynomials are not additive for this closed-open decomposition.

    The above motivic factorization reduces our study to local geometry supported on each stacky locus.
    Locally, an orbisurface $\mcX$ whose only stacky locus is a smooth divisor $D$ assumes the form $\mcX = \sqrt[r]{(S,D)}$.
    Contrary to the author's initial belief, the stacky curve components cause few difficulties:
    \begin{theorem}[proposition \ref{prop:curve-punctual-series}]
For $\mcX = \sqrt[r]{(S,D)}$, a point $I\in \Quot^\alpha\mcX$ is equivalent to a flag, whose members uniquely decompose $I$:
    \[I=\bigoplus_{i=0}^{r-1}z^iJ_i,\quad J_{r-1}(-D)\subset J_0\subset\cdots\subset J_{r-1}.\]
    Put $k=\ell(\mcO/J_{r-1})$ and $\eta_j=\ell(J_j/J_{j-1})$.
    The punctual series is
    \[ \sum_{k,\boldsymbol\eta}P(P^{(r)}_{k,\boldsymbol\eta};q) t^k\mathbf z^{\boldsymbol\eta} =F(q,t)\prod_{j=1}^{r-1}\prod_{m\ge0}\frac{1}{1-q^{2m}t^mz_j},\]
    where the $P^{(r)}_{k,\boldsymbol\eta}$ is the scheme parameterizing all such flags.
    \end{theorem}
    The global Hilbert-Chow morphism $\Quot^\alpha \sqrt[r]{(S,D)} \ra \Sym^aS$ is semismall with appropriate fibres isomorphic to $P^{(r)}_{k,\boldsymbol{\eta}}$.
    This recovers the ordinary Hilbert scheme of the coarse surface $S$.

    The remaining two possible stacky loci are either an isolated stacky point, or an intersection of stacky curves.
    With complicated interplays between all level of stacky locus, we postpone studies of orbisurfaces with intersections of stacky curves to sequel work.
    For isolated stacky point, it locally looks like $\Quot^{\bfv}[\CC^2/G]$, with a dimension vector $\bfv$ specifying the representation type of the quotient.
    The type of the Quot scheme is simplified to a finite representation of $G$ or equivalently its dimension vector $\bfv\in K_0(\CC[G])$.
    For balanced type $\bfv=n\delta$, we simplify our notation by setting $B_G(q,t):=B_{[\CC^2/G]}(q,t)$.
    A particularly notable progress made by Craw-Gammelgaard-Gyenge-Szendroi in \cite{CrawGammelgaardGyengeSzendroi2021b} is about the total cohomology of the Hilbert scheme global quotient stack $[\CC^2/G_0]$ by finite subgroup $G_0\subset SL_2(\CC)$,
which happens to agree with the crepant resolution $Y_{G_0}$ of the coarse space $\CC^2/G_0$ of our stack.
    Unfortunately, this in general is not true for finite subgroup of $G < GL_2(\CC)$, and a theory similar to Nakajima's was sought by Gusein-Zade in \cite{GuseinZadeLuengoMelle2010} but never found,
with the Betti numbers of each $\Hilb^n[\CC^2/G]$ for \textit{small} finite subgroup of $GL_2(\CC)$ computed by Bejleri-Zaimi in \cite{Bejleri_Zaimi_2023} via Bia\l{l}ynicki-Birula decompositions (later as \textbf{BB decompositions}); the word 'small' here means the only fixed point of $\CC^2$ under $G$-action is the origin.
    
    When the small group $G\not< SL_2(\CC)$, this series is different from the generating series for the minimal resolution $Y=Y_G$,
as opposed to the crepant case in \cite{CrawGammelgaardGyengeSzendroi2021b} surveyed in section \ref{sec:preliminaries}.
    Previous motivic factorization shows this is naturally a computation of the punctual fibre of the Hilbert-Chow morphisms:
    \begin{theorem}[Theorem \ref{thm:support}]
\label{intro-isolated-supp}
For every small finite subgroup $G\subset GL_2(\CC)$,
        \[B_G(q,t)=F(q,t)\cdot K_G(q,t),\]
        where $K_G(q,t)=\sum_{n,d}\dim I_n^d(G)q^dt^n$ and
        \[I_n^d(G)=\operatorname{im}\bigl(H_c^d(\Hilb^n[\CC^2/G],\QQ) \longrightarrow H^d(\Hilb^n[\CC^2/G],\QQ)\bigr).\]
        These graded dimensions are the multiplicities in the decomposition theorem for the Hilbert-Chow morphisms.
    \end{theorem}
    \begin{rmk}
        Note that a generating series was not given.
        By the Euler product factorization in \cite{stanley1986enumerative},
        we can always write such a generating series as the following infinte Euler product
        \[\prod_{m\geq 1} (\frac{1}{1-q^m})^{c_m}\]
        for constants $c_m$, but there are essential difficulties underlying the unifying solution of such power constants in terms of $G$.
        Computationally, Buryak has come up with a very nice formula for $G=C_{r,1}$ in \cite{Buryak2012} by analyzing the punctual geometry carefully.
        The author welcomes any help and suggestions towards such unifying solution.
    \end{rmk}
    
    We study these punctual fibres in the spirit of Brian\c{c}on's work \cite{Briancon1977}.
The following family $L$ of algebras is called the \textbf{flat collision algebra},
    whose description underlies the motivic computation of punctual fibres in Section \ref{subsec:collision-integration}.
    The support factorization then recovers the graded dimensions of the collision spaces;
    their intersection pairings also require the ambient embedding.
    \begin{theorem}[Proposition \ref{qfij:flat}]
    \label{intro-punctual} 
    Let $G\subset GL_2(\CC)$ be a small finite subgroup.
    A one-parameter family $L$ of $n$ distinct free $G$-orbits tending to the origin has a unique flat limit in $\Hilb^n[\CC^2/G]$.
    Every point of the principal punctual fibre arises in this way.
    Its coordinate algebra is the special fibre of the algebra generated by coordinate evaluations along the family.
    \end{theorem}

    Another important feature of Nakajima's theory is the Fock space description of total cohomology of $\Hilb^nS$ for smooth surface $S$.
    Similar theory can be obtained but with genuine new Hall-type algebras which was genuinely novel to the best of author's knowledge.
    \begin{theorem}[Proposition \ref{prop:punctual-action}]
\label{intro-hall-action}
    Let $G\subset GL_2(\CC)$ be a small finite subgroup.
    The Borel-Moore Hall algebra of balanced punctual $\CC[x,y]\rtimes G$-modules acts on
    \[\mathbb V_G=\bigoplus_{n\geq0}H^*(\Hilb^n[\CC^2/G],\QQ),\]
    preserving the collision subspace $\mathcal I_G=\bigoplus_{n,d}I_n^d(G)$.
    The series $B_G$ and $K_G$ are the graded characters of this representation and its collision submodule,
respectively.
    \end{theorem}

We now describe the organization of the paper.
Section \ref{sec:preliminaries} reviews the Quot functor of ideals,
McKay and special McKay correspondence, and the ordinary and motivic surface series.
Section \ref{sec:punctual-reduction} establishes the finite stratification of an effective surface stack,
the motivic factorization thereupon, and the initial treatment over stacky curves.
Section \ref{sec:support} proves theorem \ref{intro-isolated-supp} via BBDG's decomposition theorem of the Hilbert-Chow morphisms;
it identifies $K_G$ with compact intersection images and explains their computation using the ambient normal geometry.
Section \ref{sec:formation} develops the geometry of the equivariant punctual fibres in the spirit of Brian\c{c}on's ideas,
and computes scalar collision families;
its final subsections reconstruct higher collisions from finite jets of evaluation lattices and recover their motives by twisted integration.
Section \ref{sec:hall-action} proves theorem \ref{intro-hall-action} by constructing the punctual Hall algebra and its Borel-Moore action on the total Hilbert cohomology.
It proves preservation of the compact images, and extends the construction to $\Quot^\alpha\mcX$.
It concludes with the vacuum, hyperbolic coefficients, and finite grading bounds for resolutions.
    
    \subsection*{Notation}
    We fix the following notations:
    \begin{table}[H]
        \begin{tabular}{ll}
        $G$ & (small unless otherwise specified) finite subgroups of $GL_2(\CC)$,\\ $G_0=G\cap SL_2\CC$ & finite subgroup of $SL_2(\CC)$,\\ 
        $\Gamma=G/G_0\cong\det(G)$ & the cyclic quotient;\\ $\rho_0,\rho_1,...,\rho_s$ & irreducible representations of $G$ with $\rho_0$ trivial,\\
        $ U:=\CC^2 \curvearrowleft GL_2(\CC)\supset G$ & the defining representation of $G$, \\     $\chi = \det U$ & its determinant character; \\ 
        $\mcX\ra X$ & an orbisurface and its coarse moduli scheme,\\ 
        $X_G:= \CC^2/G$ & the (generally non-Gorenstein) affine singularity,\\ 
        $Y_G\ra X_G$ & any minimal resolution of $X_G=\CC^2/G$,\\ 
        $\mcX_G := [\CC^2/G]$ & the quotient stack;\\ 
        $\lambda=(\lambda_1\geq\cdots\geq\lambda_r>0)\vdash n$ & \\ 
        $=n^{a_n}(n-1)^{a_{n-1}}\cdots 2^{a_2}1^{a_1}$  & a partition of $n$ in two formats,\\ 
        $l(\lambda)=r,\quad |\lambda|=\sum_i\lambda_i=n$ & its length and size, \\ 
        $a= a(\lambda)=(a_1,...,a_n)$ & an ordered composition of integer $l$ equivalent to $\lambda$;\\ 
        $\alpha \in \ZZ \oplus \bigoplus_{c}(K_0(\CC[G_c])) $ & the type of an ideal of $\mcX = \Delta \coprod\mcU$\\
        & with $c$ ranging over stratum of $\Delta$ and $G_c$ the respective stabilizer.
        \end{tabular}
    \end{table}

     \subsection*{Acknowledgment}
    The author warmly thanks Jason Starr for his constant support and encouragements, and for bringing this question to the author's attention.
     The author also extends his gratitude to Weiqiang Wang for helpful conversations at ``\href{https://representationtheory.ca/mmrt2026/}{Modern Methods in Representation Theory}" Conference last summer which steers the author out of futile chase in Heisenberg algebra.
     
     This paper is partially supported by the ICM 2026 Graduate Student and Early Career Travel Grant program, funded by NSF award No. 2620718.

\section{Background and Preliminary}
\label{sec:preliminaries}

All spaces are over $\CC$, and cohomology has rational coefficients.
For a finite subgroup $G\subset GL_2(\CC)$, put $R=\CC[x,y]$,
let $U=\langle x,y\rangle$ be its coordinate representation, and put $\chi=\det U$.
Write $\rho_0,\ldots,\rho_s$ for its irreducible representations, with $\rho_0$ trivial, and set
\[ d_i=\dim\rho_i,\qquad \delta=(d_0,\ldots,d_s),\qquad |\mathbf v|=\sum_i d_iv_i.
\]
As $\CC[G] =\oplus \rho_i^{\oplus d_i}$, $\delta$ is the \textbf{dimension vector} of $\CC[G]$.
We distinguish the $R$-colength $|\mathbf v|$, balanced length $n$ when $\mathbf v=n\delta$, and the global numerical type $\alpha$ introduced later in section \ref{subsec:alpha-type}.

\begin{defn}
\label{def:small}
A finite subgroup of $GL_2(\CC)$ is small if it has no non-identity pseudo-reflections.
Equivalently, every non-identity element acts without a fixed point on $\CC^2\setminus\{0\}$.
\end{defn}
Finite-order linear transformations are diagonalizable, so a non-identity transformation fixing a nonzero vector in dimension two is exactly a pseudo-reflection\footnote{whose fixed locus is of codimension 1.}.
Hence the condition of being small concerns every element in the group, not merely the common fixed subspace of the group.

\subsection{Quot functors and equivariant quotients}
\label{subsec:ideal-hilbert-functors}
Let $\mcX$ be a smooth orbisurface of finite type over $\CC$.
Following Olsson-Starr's work \cite{olsson2003quot}, our Quot functor assigns to $T$ the quotients
    \[ \mcO_{\mcX\times T}\twoheadrightarrow\mcQ=\mcO_{\mcX\times T}/\mathcal J, \]
which are finitely presented and flat with proper support over $T$ and with zero-dimensional geometric fibres.
This is a functor parameterizing ideals with certain properties.
We write $\Quot^\alpha\mcX$ after fixing the type $\alpha$ defined in section \ref{subsec:alpha-type}.

The general representability theorem is due to Olsson-Starr \cite{olsson2003quot}.
Nonetheless, motivic factorization rooted in the stratification of the stack $\mcX$ makes the following specialized definition very useful,
which is the orbifold Quot scheme given by Craw-Gammelgaard-Gyenge-Szendr\H{o}i \cite{Craw_Gammelgaard_Gyenge_Szendroi_2021a}.

\begin{defn}
\label{def:ideal-functor}
For a finite group $G$ acting on $\CC^2$ and a nonnegative dimension vector $\bfv$, the functor $\Quot^{\bfv}[\CC^2/G]$ parametrizes $G$-invariant ideals $\mcJ\subset\mcO_T[x,y]$ whose quotient is locally free over $T$ and satisfies
    \[ \rk\Hom_G(\rho_i,\mcO_T[x,y]/\mathcal J) = v_i.\]
\end{defn}

\begin{prop}
\label{prop:representability}
The scheme $\Quot^{\mathbf v}[\CC^2/G]$ is the open-and-closed multiplicity locus in $\Hilb^{|\mathbf v|}(\CC^2)^G$.
In particular it is a quasi-projective scheme, with its tautological equivariant quotient.
\end{prop}
\begin{proof}
A fixed Hilbert point is a $G$-invariant ideal;
equality of ideals leaves no additional equivariant structure to choose.
On the fixed scheme the tautological quotient decomposes into the direct summands $\rho_i\otimes\Hom_G(\rho_i,\mcQ)$.
Averaging over $G$ makes the isotypic projections commute with base change, so their ranks are locally constant.
Descent data identifies invariant ideals with ideals on the quotient stack.
The fixed scheme is closed in $\Hilb^{n|G|}\CC^2$.
\end{proof}

The same argument applies to a smooth surface with a finite-group action.
General surface stacks are handled by these charts at the support of stacky loci;
section \ref{subsec:alpha-smoothness} supplies both the precise global representability statement and its smoothness proof.

\subsection{McKay and special McKay correspondence}

 In 1980, McKay observed a natural one-to-one correspondence that every Kleinian singularities has the form of a quotient by finite groups,
$\CC^2/G_0$ for a finite subgroup $G_0$ of $SL_2(\CC)$; it is very surprising even to McKay himself that this correspondence was discovered so late in 1980 \cite{McKay1980}.
% ..., as this seems to be foreshadowed by Klein himself in the Erlangen program.
These singularities admit crepant resolutions\footnote{The admission of a crepant resolution is an equivalent definition of Klein singularities on a surface.}
$Y_{G_0}\xrightarrow{\pi} \CC^2/G$, which means $\pi^*K_X = K_Y$.
Further study \cite{Gonzalez-Sprinberg_Verdier_1983} revealed a bijective correspondence between irreducible components of the exceptional fibre of the resolutions and irreducible representations of $G_0$.
Via standard quiver operation as doubling and framing (see standard textbook on quiver theory such as \cite{kirillov2016quiver} for details),
the dual graphs of the exceptional fibre form \textbf{McKay quiver}.
Moreover, we can reconstruct the resolution $Y_{G_0}$ and its Hilbert schemes of $n$ points, $\Hilb^n Y_{G_0}$,
as the Nakajima's quiver varieties of the \textbf{McKay quiver};
this result was established in this form by Kuznetsov in \cite{Kuznetsov2007} (with earlier pioneering works \cite{HaimanHilbertQuiver,VaragnoloVasserot1999,Wang2002}).
Even further, because of the crepant nature of this resolution, Craw and his collaborators found the following isomorphisms in \cite{CrawGammelgaardGyengeSzendroi2021b}
\begin{equation}
    \label{eq:crepant}
H^*(\Hilb^nY_{G_0},\QQ)\cong H^*(\Hilb^n([\CC^2/G_0],\QQ)),
    \end{equation}
successfully relating the DM stack $[\CC^2/G_0]$ to the crepant resolution of its (singular) coarse moduli space.

Driven by curiosity, and in our case driven by the natural setting of stabilizers occurring in orbisurfaces,
we want to extend McKay type links to quotient stack $\mcX_G=[\CC^2/G]$ and its coarse moduli space $X_G=\CC^2/G$ with $G$ a small finite subgroup of $GL_2(\CC)$.
The singularities of these coarse moduli spaces turn out to admit no crepant resolution,
but the McKay-thematic link between geometry of their minimal resolutions $Y_G$ and certain subset (called \textbf{special}) of irreducible representations of $G$ is still present through the \textbf{special McKay correspondences},
where the special representations were first introduced by Wunram in \cite{wunram1988reflexive} and the correspondence was established in $A$-type by \cite{Ito2002} and in all types by Wemyss and his collaborators all of whose work are summarized in \cite{Wemyss2011},
with further comparison between derived categories of $Y_G$ and $\mcX_G$ obtained by Ishii-Ueda in \cite{Ishii_Ueda_2013}.
The most important distinction between the crepant case and discrepant case is that equation (\ref{eq:crepant}) does not hold in the discrepant case.

We give a quick overview here of the special McKay correspondence, readers who are interested in a detailed or historical account of this theory should consult \cite{Wemyss2011} or \cite{Ishii_Ueda_2013}.
As was documented in \cite[section 5]{Wemyss2011}, there is a finite list of small finite subgroup $G<GL_2(\CC)$:
denoted by types $\AAA(=C_{r,a}), \mbb{D}_{n,q}, \mbb{T}_m, \mbb{O}_m, \mbb{I}_m$ extending the finite subgroup of $SL_2(\CC)$ of type $A,D,E_6,E_7$ and $E_8$ respectively with certain conditions imposed on the parameters $a,q,m$ causing the discrepancy.
For all small groups $G$, put $N_\rho=(R\otimes\rho^\vee)^G$ and let $\widetilde N_\rho=f^*N_\rho/\text{torsion}$ be its full locally free sheaf on $Y_G$.
An irreducible representation is \textit{special} if
    \[ H^1(Y_G,\widetilde N_\rho^\vee)=0.\]
In particular for $G=C_{r,a} = <
\begin{pmatrix}
    \zeta_r & \\ & \zeta_r^a
\end{pmatrix}
>$ with $(a,r)=1$,
the monomial $xy$ corresponds to representation $\rho_{1+a}\neq \rho_0$, which is non-special.
The special McKay correspondence matches all the nontrivial special representations one-to-one with all the irreducible exceptional curves $E_\sigma$,
characterized by\footnote{The reader should be able to tell the Kronecker delta apart from the dimension vector delta defined previously.}
$c_1(\widetilde N_\rho)\cdot E_\sigma=\delta_{\rho\sigma}$.
In the crepant case, every irreducible representation is special.

\begin{rmk}
    \label{rmk:dim_vector}
The skew group ring $R\rtimes G$ is of global dimension $2$ in the crepant case, and of global dimension 3 in the non-crepant case,
with the information of $\Ext^2$ documented by dashed arrow $\rho_0 \dashrightarrow \chi$ and transported by symmetry to all nodes $\rho_i\dashrightarrow \rho_i\otimes \chi$.
Such information is also organized by Ishii-Ueda in their McKay quiver\footnote{The notion of quiver is extended so dashed arrow is allowed to capture syzygies.}, which gives the usual McKay quiver for the crepant case.
The dimension of appropriately defined quiver variety of dimension $\bfv$ over the McKay quiver for non-crepant $G$ gives the proof to lemma \ref{lem:smooth-dimension}.
\end{rmk}

\subsection{Surface series and additive realizations}
\label{subsec:surface_series}
For a variety $Z$, put $P(Z;q)=\sum_d\dim H^d(Z;\QQ)q^d$.
The variable $q$ records actual cohomological degree.
For a smooth quasi-projective surface $S$, the classical formula in \cite{Nakajima_1999,Gottsche_2000} is
\begin{equation*}
 F_S(q,t):=\sum_{n\ge0}P(S^{[n]};q)t^n =\prod_{m\ge1} \frac{(1+q^{2m-1}t^m)^{b_1(S)}(1+q^{2m+1}t^m)^{b_3(S)}} {(1-q^{2m-2}t^m)^{b_0(S)}(1-q^{2m}t^m)^{b_2(S)} (1-q^{2m+2}t^m)^{b_4(S)}}.
\end{equation*}
It is the character of the graded-symmetric algebra on copies of $H^*(S)$ in lengths $m\ge1$, shifted by degrees $2m-2$.
Odd classes give exterior powers and hence the numerator.

The motivic counterpart uses $\LL=[\AAA^1]$ and $\zeta_Z(T)=\sum_{a\ge0}[\Sym^a Z]T^a$:
\begin{equation*}
 \sum_{n\ge0}[S^{[n]}]t^n =\prod_{m\ge1}\zeta_S(\LL^{m-1}t^m) =\left(\prod_{m\ge1}\frac{1}{1-\LL^{m-1}t^m}\right)^{[S]}.
\end{equation*}
The last exponent is the geometric power structure: its coefficient parametrizes finite configurations of distinct support points,
each decorated by a punctual object \cite{GuseinZadeLuengoMelle2010}.
It is not an ordinary power of a polynomial.

For non-proper support strata an additive realization is necessary.
We use the Deligne-Hodge polynomial with compact support
    \[ E_c(Z;u,v)=\sum_{k,a,b}(-1)^k h^{a,b}(\operatorname{Gr}_{a+b}^W H_c^k(Z))u^av^b.\]
It respects open-closed decomposition and products, sends $\LL$ to $uv$,
and carries the geometric power structure to its Hodge realization.
If $Z$ is proper and $H^k(Z)$ admits a pure Hodge structure of weight $k$, then $P(Z;q)=E_c(Z;-q,-q)$.
In particular this holds for smooth projective $Z$ and for the proper pure punctual cores below.
Without such purity, Section \ref{sec:support} obtains ordinary Betti numbers from closed intersection-complex supports.

For comparison with the isolated quotient problem, let $Y_G\ra\CC^2/G$ be the minimal resolution and put $e=b_2(Y_G)$.
Its only nonzero Betti numbers are $b_0=1$ and $b_2=e$, so
\begin{equation}
\label{eq:resolution-series}
F_{Y_G}(q,t)=F(q,t)\cdot (\prod_{m\ge1}\frac{1}{1-q^{2m}t^m})^{e}.
\end{equation}
In the crepant case the Hilbert-scheme comparison \cite[Corollary 4.13]{Kuznetsov2007} identifies this with the orbifold Betti series.
The Betti number calculation for non-crepant cyclic groups $C_{r,a}$ of Bejleri-Zaimi is deferred to Section \ref{subsec:balanced-partitions},
after the reduction to isolated quotient charts.

\section{Punctual Reduction of the Stacky Surface}
\label{sec:punctual-reduction}

\subsection{The finite stratification of the stacky locus}
\label{subsec:finite-stacky-locus} 
We first separate ineffective isotropy from the effective surface geometry.
If $\mathcal G\ra\mathcal Y$ is a finite gerbe, pullback identifies their Quot functors:
locally the map is $BK\times V\ra V$, and $K$ acts trivially on $\mcO_V$, hence on every ideal and quotient.
Such ineffective descent is unique and preserves flatness and proper support.
For a smooth surface stack its ineffective rigidification has this form.
Thus the ideal geometry can be studied on the effective rigidification;
representations for the original stack are obtained by inflation.
In particular an ideal quotient is trivial under ineffective isotropy, and cannot be a positive multiple of its regular representation.
This explains why we assume all orbisurfaces involved to be effective stacks.

The finite type assumption yields the following finite stratification.
\begin{prop}
\label{prop:finite-stacky-locus}
Let $\mcX$ be an effective smooth separated DM surface of finite type, with coarse space $X$.
There is a finite stratification
\begin{equation*}
 X=X^\circ\sqcup\coprod_{a=1}^sD_a^\circ\sqcup P,
\end{equation*}
where $X^\circ$ has trivial stabilizer, each $D_a^\circ$ is a smooth curve with constant cyclic stabilizer $\mu_{r_a}$ acting with weights $(1,0)$,
and $P$ is finite.
The set $P$ includes isolated stacky points, intersections and singular points of the stacky curves,
and their stabilizer-jump points.
\end{prop}
\begin{proof}
The inertia morphism is proper, quasi-finite, and unramified,
hence finite and unramified.
Its identity section is both open and closed.
The remaining inertia therefore has closed image.
This is the stacky locus.
It is Noetherian and misses the generic point, so it has finitely many irreducible components, all of dimension at most one.

At a geometric point the formal chart is a quotient of $\CC[[x,y]]$ by its finite stabilizer.
Averaging lifts of the cotangent space linearizes this action.
At a general point of a stacky curve the group fixes its tangent line pointwise.
Effectiveness makes its action on the normal line faithful,
so the group is cyclic; choosing its normal character gives weights $(1,0)$.
Shrinking each of the finitely many curves makes the stabilizer and local representation type constant.
The removed closed subsets of the curves are finite.
Add their points to the zero-dimensional components to obtain $P$.
\end{proof}

This gives three local stacky pictures: the stacky curves, the isolated stacky point and intersection of stacky curves.
The key observation for the stacky curve case is that the quotient of the reflection chart is smooth:
$\CC[[x,y]]^{\mu_r}=\CC[[x^r,y]]$.
At an isolated point the effective group has no non-identity pseudo-reflections; otherwise a fixed line would produce a stacky curve.
This third local problem is delayed for a sequel work;
but initial inspection at the example of independent roots with stabilizer $\mu_r\times\mu_s$ at their crossing have smooth coarse ring $\CC[[x^r,y^s]]$, so the theory around intersections are closer to stacky curves than isolated stacky points.

The ordinary chart leads to Brian\c{c}on's punctual Hilbert scheme; a stacky line chart adds a cyclic flag to its local algebra; an isolated chart constrains both multiplication operators by the stabilizer.
The collision of the supporting points is precisely the passage to these punctual algebras.

\subsection{Length and representation type}
\label{subsec:alpha-type}
For a zero-dimensional quotient define its type by
\begin{equation}
\label{eq:alpha-euler}
\alpha_{\mcQ}:K_0^{\mathrm{vb}}(\mcX)\longrightarrow\ZZ,
 \qquad [\mathcal E]\longmapsto \chi(\mcX,\mcQ\otimes\mathcal E).
\end{equation}
The possible types form a submonoid of the image of the zero-dimensional coherent $K$-group in $\Hom(K_0^{\mathrm{vb}}(\mcX),\ZZ)$.
Concretely, by our finite stratification, $\alpha $ is a class in the following lattice of length and correction representation at each component of the stacky loci
\[\ZZ \oplus \bigoplus_{c}(K_0(\CC[G_c])).\]
We call $\alpha$ \textbf{admissible} when it occurs for an ideal quotient,
and define $\Quot^\alpha\mcX$ by equality $\alpha= \chi(\mcX,\mcQ\otimes -)$ on every geometric fibre.
This is the representation-sensitive Hilbert polynomial of Olsson-Starr.
It is constant on connected flat families with proper support, so its loci are open and closed \cite[Lemma 4.3]{olsson2003quot}.

Here is its concrete relation to local characters.
At each coarse support point $p$, let $V_p$ be the finite quotient representation in a chart with stabilizer $G_p$.
Put \begin{equation}
\label{eq:local-signed-defect}
n_p=\dim V_p^{G_p},\qquad \beta_p=[V_p]-n_p[\CC[G_p]],\qquad n=\sum_pn_p.
\end{equation}
The trivial-representation coefficient of $\beta_p$ is zero.
If $V_p=\bigoplus_i\rho_i^{v_{p,i}}$, its other coefficients are $v_{p,i}-n_p\dim\rho_i$, and may be negative.
Since taking $G$-invariant is an exact functor, a filtration by powers of the maximal ideal gives
\begin{equation*}
 \alpha_{\mcQ}([\mathcal E]) =n\operatorname{rk}\mathcal E +\sum_p\dim(\beta_p\otimes\mathcal E|_p)^{G_p}.
\end{equation*}
The last expression is interpreted additively for virtual representations.
Thus $n=\alpha_{\mcQ}([\mcO])$ is a nonnegative integer called the \textbf{support length}, equal to the ordinary length on $\mcU$,
and the balancing condition forces type $n$ to contributes balanced representations $n_p\operatorname{Reg}(G_p)$ at every stabilizer type.
Such support length is different from both the local $R$-colength and rational stack degree given by:
\begin{equation*}
 \deg_{\mcX}\mcQ =\sum_p\frac{\dim V_p}{|G_p|} =n+\sum_p\frac{\dim\beta_p}{|G_p|}.
\end{equation*}
There is a corresponding coarse support map.
Tameness makes $\pi_*$ exact and compatible with base change, so $\pi_*\mcQ$ is a flat quotient of $\mcO_X$ of ordinary length $n$.
Its Hilbert-Chow cycle defines a Hilbert-Chow morphism (or \textbf{HC morphism} for short)
\[ h_\alpha:\Quot^\alpha\mcX\longrightarrow\operatorname{Sym}^nX,
 \quad \mcQ\longmapsto\sum_pn_p[p].\]

For stacky curve cases, Mochizuki has developed the right tool in his early-year work \cite{MochizukiParabolic2002} on parabolic Hilbert schemes.
For such a parabolic type one can refine this target by recording the support cycles of the flag jumps separately.
For example, at a reflection point of order $r$ the reduced residual gerbe has $V_p=\rho_0$.
Its integer length is $1$,
its stack degree is $1/r$, and its defect is $-\rho_1-\cdots-\rho_{r-1}$.
By contrast the ideal $(x^r,y)$ has regular representation, integer length $1$,
stack degree $1$, and zero defect.

The local representation groups must be glued, not put into an unqualified direct sum.
When a point with stabilizer $H$ specializes into a chart with stabilizer $G\supset H$, its representation contributes $\operatorname{Ind}_H^G V$.
This is the representation on the translates of its local algebra.
Frobenius reciprocity gives
\[ \dim(\operatorname{Ind}_H^G V)^G=\dim V^H,
 \qquad \operatorname{Ind}_H^G\operatorname{Reg}(H) =\operatorname{Reg}(G),
\]
so integer length and signed defects obey the same incidence rule.
Several incoming supports contribute the sum of their induced representations.
Along a stratum the character data are transported with its stabilizer local system.
Equation \eqref{eq:alpha-euler} records their intrinsic global numerical image, including these relations.

\begin{rmk}[Numerical and local balance]
\label{rmk:numerical-local-regular}
The numerical regular class is $\alpha_{n,\mathrm{reg}}=n\operatorname{rk}$.
Local balance means $\beta_p=0$ separately at every support point.
It implies numerical balance.
Recovering every local character from a global functional requires a separate detection statement; a generating bundle or the resolution property alone does not assert independent detection at each support point.
Indeed the character of a vector bundle along a connected constant-stabilizer stratum is locally constant,
so the pairing records aggregate defects there.
We retain explicit chart representation vectors whenever a calculation needs this finer information.
For $[\CC^2/G]$, testing with the constant bundles $\rho_i^\vee$ recovers the full quotient dimension vector,
and $\alpha=n\operatorname{rk}$ is exactly total charge $n\delta$.
General support formulas state whether they concern this numerical type or the separately imposed local regular sector.
\end{rmk}

The length used in a parabolic flag below has yet another useful normalization.
If its character colengths are $d_i=k+\sum_{j>i}\eta_j$, then the integer \eqref{eq:local-signed-defect} is $n=d_0=k+\sum_j\eta_j$.
We use $k$ for the base-ideal colength in that flag formula; its variables must not silently be identified with the numerical $n$.

\subsection{Representability, smoothness, and formal locality}
\label{subsec:alpha-smoothness}
\begin{theorem}
\label{thm:global-hilbert-smooth}
For every numerical type $\alpha$, the Quot functor $\Quot^\alpha\mcX$ is a smooth separated algebraic space locally of finite type.
If $\mcX$ has quasi-projective coarse space, it is a quasi-projective scheme; if the coarse space is projective,
it is projective.
Empty types are allowed.
The same smoothness assertion holds after imposing locally constant equivariant multiplicity types.
\end{theorem}
\begin{proof}
Olsson-Starr's Theorem 1.1 represents the Quot functor by a separated algebraic space; fixing its Hilbert polynomial is open and closed.
A regular stack of dimension two with quasi-finite diagonal is a global quotient \cite[Example 2.17]{EHKV2001}.
Olsson-Starr's Theorem 4.4 therefore gives the stated scheme and projectivity assertions \cite{olsson2003quot}.

For smoothness, take an ideal quotient with distinct coarse support points $p$.
In a deformation over an Artinian local base, the idempotents separating these supports lift uniquely.
Its deformation functor is therefore the product of the deformation functors in the completed charts $[\operatorname{Spec}\CC[[x,y]]/G_p]$.
Linearizing each chart identifies its factor with the fixed deformation functor in an ordinary surface Hilbert scheme.
That scheme is smooth by Fogarty's theorem \cite{Fogarty_1968}.
A finite group acting on a smooth complex formal germ has an equivariant system of parameters: averaging a section of the cotangent-space quotient gives such parameters.
Its fixed germ is thus a power-series ring on the invariant tangent space.
Each factor, and hence their product, is formally smooth.
The locally finite-type representing space is consequently smooth.
Prescribing multiplicity ranks selects open-and-closed loci and preserves smoothness.
\end{proof}

\begin{cor}[Formal locality]
\label{prop:locality}
At $J_p=\ker( \mcO_{\mcX}\twoheadrightarrow\mcQ)$ the formal Quot germ is the product of the equivariant formal Quot germs at its distinct supports.
Its tangent space is
\[ \bigoplus_p\Hom_{\CC[[x,y]],G_p}(J_p,Q_p).
\]
The dimensions may differ between connected components.
\end{cor}
\begin{proof}
This is the support splitting in the preceding proof, together with the usual tangent-space description for ideals.
\end{proof}

This smoothness theorem concerns $\Quot^\alpha \mcX$.
Fixing the support $Z$ cuts out a punctual scheme $\Quot^\alpha_Z\mcX$,
which may be singular or reducible.
Formal locality separates disjoint supports and it reduces the geometric question to the punctual fibres, 
whose geometries are the Brian\c{c}on-type questions taken up in section \ref{sec:formation}.

\subsection{Dimensions in quotient charts}
For finite-dimensional representations $V,W$, write $\langle V,W\rangle_G=\dim\Hom_G(V,W)$, and define
\[E(\mathbf v,\bfw) =\langle V_{\mathbf v},V_{\bfw}\rangle_G -\langle U\otimes V_{\mathbf v},V_{\bfw}\rangle_G +\langle\chi\otimes V_{\mathbf v},V_{\bfw}\rangle_G,
\qquad V_{\mathbf v}=\bigoplus_i\rho_i^{v_i}.\]
Recall that subscripts of dimension vectors in $K_0(\CC[G])$ are irreducible representations of $G$,
so amusingly we are actually abusing the notation with $v_0$.

\begin{lem}[\textit{cf.} remark \ref{rmk:dim_vector}]
\label{lem:smooth-dimension}
Every nonempty $\Quot^{\mathbf v}[\CC^2/G]$ is smooth of pure complex dimension
\begin{equation}
\label{eq:dimension}
D(\mathbf v)=v_0+v_{\chi^{-1}}-E(\mathbf v,\mathbf v).
\end{equation}
Moreover,
\[ D(n\delta)=2n,\qquad D(\mathbf v+n\delta)=D(\mathbf v)+2n \]
whenever the indicated charge vectors are nonnegative.
\end{lem}

\begin{proof}
Let $Q=R/J$ be the quotient at a point.
We essentially compute the tangent space dimension with equivariant settings which was already explicit in \cite{Ishii_Ueda_2013}.
Multiplication gives
\[ \Hom_{R\rtimes G}(Q,Q)\cong Q^G \cong\Hom_{R\rtimes G}(R,Q).\]
The functor $\Hom_{R\rtimes G}(R,-)=(-)^G$ is exact.
Applying $\Hom_{R\rtimes G}(-,Q)$ to exact sequence 
\[0\ra J\ra R\ra Q\ra 0\] consequently gives
\[ T_J\Quot^{\mathbf v}[\CC^2/G]=\Hom_{R\rtimes G}(J,Q) \cong\Ext^1_{R,G}(Q,Q).\]
Equivariant duality on the smooth affine plane (\textit{cf.} remark \ref{rmk:dim_vector}) gives
\[ \Ext^2_{R\rtimes G}(Q,Q) \cong\Hom_{R\rtimes G}(Q,Q\otimes\chi)^*,\]
of dimension $v_{\chi^{-1}}$.
The two-variable Koszul resolution computes the alternating extension dimension as $E(\mathbf v,\mathbf v)$, proving \eqref{eq:dimension}.
This dimension depends only on $\mathbf v$.

Tensoring any representation with $\CC[G]$ multiplies the regular representation by its dimension.
Since $1-\dim U+\dim\chi=1-2+1=0$, we obtain $E(\delta,\bfv)=E(\bfv,\delta)=0$.
Also $d_0=d_{\chi^{-1}}=1$.
Substituting in \eqref{eq:dimension} give the dimension for unbalanced dimension vector.
\end{proof}

\subsection{Pure stacky curves and parabolic punctual fibres}
\label{qfc:subsec:curve-series}

Let $D\subset S$ be a smooth divisor and $\mcX=\sqrt[r]{(S,D)}$, with $r\ge2$.
A root chart has rings $B=\CC[[x,y]]$ and $A=B[z]/(z^r-x)$, where $z$ has character given by multiplication by $r$-th root of unity $\zeta_r$ and $B$ is invariant.
A $\mu_r$-invariant ideal $I\subset A$ has a unique decomposition
\begin{equation}
\label{curve:eq:flag}
I=\bigoplus_{i=0}^{r-1}z^iJ_i,
 \quad xJ_{r-1}\subset J_0\subset J_1\subset\cdots\subset J_{r-1}\subset B.
\end{equation}
These $J_i$'s are ideals of $B$ and the inclusions in the flag are compatible with multiplication by $z$.
The first inclusion globalizes to $J_{r-1}(-D)\subset J_0$.
Conversely such a flag defines an ideal of the root stack, so the two constructions are inverse operations in flat families.

Put
\[ k=\ell(\mcO_S/J_{r-1}),\qquad \eta_j=\ell(J_j/J_{j-1})\ge0\quad(1\le j<r),\]
and write $H^{(r)}_{k,\boldsymbol\eta}(S,D)$ for this flag space.
Its character multiplicities, invariant length, and signed defects are
\begin{equation}
\label{curve:eq:gradings}
v_i=k+\sum_{j>i}\eta_j,\qquad n=v_0=k+|\boldsymbol\eta|,\qquad \delta_i=v_i-v_0=-\sum_{j\le i}\eta_j.
\end{equation}
In particular, the balancing condition of the Hilbert scheme is equivalent to $\bfv=0= \boldsymbol{\eta}$, and $\Hilb^n\mcX \cong H^{(r)}_{n,\bold{0}} (S,D)$.

\begin{prop}[Parabolic description]
\label{curve:prop:parabolic}
The space $H^{(r)}_{k,\boldsymbol\eta}(S,D)$ is smooth of dimension $2k+|\boldsymbol\eta|$.
If all gaps vanish, it is $S^{[k]}$.
\end{prop}
\begin{proof}
For $J=J_{r-1}$, the flag gives successive quotients of $J/J(-D)$:
\[ J/J(-D)\longrightarrow J/J_0\longrightarrow\cdots \longrightarrow J/J_{r-2}\longrightarrow0.
\]
Their successive kernels have lengths $\eta_1,\ldots,\eta_{r-1}$.
This identifies the functor with the finite-jump parabolic Hilbert functor of \cite{MochizukiParabolic2002}.
Its smoothness theorem gives the dimension.
Alternatively, smoothness follows from the fixed-locus argument above.
Equal colengths force every inclusion in \eqref{curve:eq:flag} to be equality, proving the final assertion.
\end{proof}

Given point of $D$, denote the punctual flag fibre by $P^{(r)}_{k,\boldsymbol\eta}$.
It is projective, and the local geometry is independent of the chosen point.

\begin{prop}[Punctual curve series]
\label{prop:curve-punctual-series}
The fibres $P^{(r)}_{k,\boldsymbol\eta}$ have even, pure Tate cohomology, and
\begin{equation}
\label{qfc:eq:refined-punctual}
\sum_{k,\boldsymbol\eta} P(P^{(r)}_{k,\boldsymbol\eta};q)t^k\mathbf z^{\boldsymbol\eta} =F(q,t)\prod_{j=1}^{r-1}\prod_{m\ge0} \frac{1}{1-q^{2m}t^mz_j}.
\end{equation}
For each fixed gap vector,
\begin{equation*}
 \sum_{k\ge0}P(P^{(r)}_{k,\boldsymbol\eta};q)t^k =F(q,t)\prod_{j=1}^{r-1}\prod_{a=1}^{\eta_j} \frac{1}{1-q^{2a}t^a}.
\end{equation*}
\end{prop}
\begin{proof}
The parabolic punctual affine paving in \cite[Section~3]{MochizukiParabolic2002} has ordinary modes of charge $(m,\mathbf0)$ and dimension $m-1$,
$m\ge1$, and defect modes of charge $(m,\mathbf e_j)$ and dimension $m$,
$m\ge0$.
Its motivic product becomes \eqref{qfc:eq:refined-punctual} under $\LL\mapsto q^2$:
projectivity and the closed filtration by affine cells give one class in degree $2a$ for each cell of dimension $a$,
and no odd cohomology.
For $h=q^2t$, successive differences in a nondecreasing sequence of $d$ nonnegative integers give
\[ [z^d]\prod_{m\ge0}\frac{1}{1-h^mz} =\prod_{a=1}^d\frac{1}{1-h^a}.
\]
Apply this identity to each $z_j$ to obtain the second formula.
\end{proof}

To use the invariant length grading replace each $z_j$ by $tz_j$.
A defect mode then has invariant length $m+1$, as \eqref{curve:eq:gradings} requires.
For $\boldsymbol\eta=0$ and $k>0$, the fibre is the ordinary Brian\c{c}on fibre, irreducible of dimension $k-1$.
For $|\boldsymbol\eta|>0$, the top cells are dense and all components have dimension $k$; their number is
\begin{equation*}
 [t^k]\prod_{j=1}^{r-1}\prod_{a=1}^{\eta_j}\frac{1}{1-t^a}.
\end{equation*}
Indeed a positive ordinary mode loses two degrees relative to twice its length, so only defect modes contribute to degree $2k$.
For example $r=2$, $(k,\eta_1)=(2,2)$ gives $P(P^{(2)}_{2,(2)};q)=1+2q^2+2q^4$ and two irreducible components.
The root condition therefore retains an explicit enumeration while allowing failure of punctual irreducibility.

For moving support, put $\zeta_T(w)=\sum_{a\ge0}[\Sym^aT]w^a$.
The parabolic theorem gives the motivic formula
\begin{equation*}
 \sum_{k,\boldsymbol\eta}[H^{(r)}_{k,\boldsymbol\eta}] t^k\mathbf z^{\boldsymbol\eta} =\prod_{m\ge1}\zeta_S(\LL^{m-1}t^m) \prod_{j=1}^{r-1}\prod_{m\ge0}\zeta_D(\LL^mt^mz_j).
\end{equation*}
Its ordinary Betti counterpart is the surface factor $F_S(q,t)$ multiplied by one curve family for each $1\le j<r$:
\[ \prod_{j=1}^{r-1}\prod_{m\ge0} \frac{(1+q^{2m+1}t^mz_j)^{b_1(D)}} {(1-q^{2m}t^mz_j)^{b_0(D)} (1-q^{2m+2}t^mz_j)^{b_2(D)}}.
\]
Section \ref{subsec:curve} proves this assertion directly by the semismall Hilbert-Chow morphism, including for quasiprojective $S$.
The proof identifies the exact supports and explains why a reducible punctual fibre does not create a new point-supported factor at every point of $D$.
This is the first instance of the distinction between punctual cohomology and a compact intersection image,
which becomes essential at isolated stacky points.
When divisors intersect or a stabilizer jumps, one must replace the single flag by a coupled local ideal problem.
A sequel paper treats these exceptional loci separately.

\subsection{The removal of reflections}

Recall that our definition of $\Hilb^n\mcX$ forces its geometric point, the ideal sheaf $\mcI$ whose quotient has local algebra at each coarse support point $p$ isomorphic to balanced representation $n_p\CC[G_p]$.
Under this assumption, we can remove the reflection stabilizers.

\begin{prop}[Reflection descent]
\label{prop:reflection}
Let $G\subset GL(V)$, $V=\CC^2$, be finite and faithful.
Let $R_G$ be the normal subgroup generated by pseudo-reflections,
$W=V/R_G$, and $\overline G=G/R_G$.
Extension and contraction of ideals give
\[ \Hilb^n[V/G] \cong\Hilb^n[W/\overline G].
\]
The surface $W$ is an affine plane, and $\overline G$ acts,
in suitable coordinates, without pseudo-reflections.
\end{prop}
\begin{proof}
Put $A=\CC[V]$, $B=A^{R_G}$.
The reflection theorem makes $B$ polynomial and $A$ free of rank $|R_G|$ over $B$.
For a regular quotient $M=A/I$, exact invariants give $M^{R_G}=B/J\cong n\operatorname{Reg}_{\overline G}$.
The surjection $A\otimes_B(B/J)\ra M$ is between spaces of the same dimension $n|G|$, hence an isomorphism.

Conversely choose a graded $G$-stable lift $C$ of the coinvariant algebra $A/B_+A$; multiplication identifies $A=B\otimes C$,
and $C|_{R_G}$ is regular.
If $N\cong n\operatorname{Reg}_{\overline G}$, the character of $C\otimes N$ vanishes outside $R_G$ by the second factor,
and on every non-identity element of $R_G$ by the first.
Its dimension is $n|G|$, so it is regular.
In families invariants commute with base change; the same surjection is between locally free modules of equal rank.
These inverse operations therefore work in families and after completion.
The residual action is small by the canonical-stack description of the reflection quotient.
\end{proof}

Let $\mcX^{\mathrm{can}}$ be the smooth canonical stack of the coarse quotient surface $X$; it is a scheme off the finite singular set.
The canonical morphism $\mcX\ra \mcX^{\mathrm{can}}$ is locally the reflection quotient above.
Pullback and invariant contraction also preserve local balance when support moves: for $v\in V$ over $w\in W$ the stabilizers fit into $1\ra (R_G)_v\ra G_v\ra\overline G_w\ra1$.
Surjectivity follows by adjusting a lift by an element of $R_G$.
The quotient of the smooth local chart by $(R_G)_v$ is a neighborhood in the smooth surface $W$.
The converse reflection theorem therefore makes $(R_G)_v$ a reflection group, so the same coinvariant argument applies in completed local coordinates.
These canonical operations agree on overlaps, giving
\begin{equation}
\label{eq:canonical-reduction}
\Hilb^n(\mcX) \cong\Hilb^n(\mcX^{\mathrm{can}}).
\end{equation}
Thus all residual punctual problems in this moduli have small groups.
For general type $\alpha$ the parabolic defects remain, and \eqref{eq:canonical-reduction} is not asserted for $\Quot^\alpha$.

\subsection{Assembly by numerical type}
\label{subsec:motivic-assembly}
Let $\Lambda_{\mcX}$ be the image of the zero-dimensional coherent $K$-group under the Euler pairing,
and put $u([\mathcal E])=\operatorname{rk}\mathcal E$.
Write $\mathbf x^\alpha$ for its class monomials.
We use a completion graded by a positive length in which each coefficient under consideration is a finite sum.
The following identity is the additive form of support assembly.
For the strata $U,D_a^\circ,\{p\}$ of Proposition \ref{prop:finite-stacky-locus}, let $\iota_p:R(G_p)\ra\Lambda_{\mcX}$ send a local representation to its Euler-pairing class.
Let $\beta_{a,j}$ denote the class of the initial character interval $1+\varepsilon_a+\cdots+ \varepsilon_a^{j-1}$ at a point of $D_a^\circ$; it is independent of the point on a connected stratum.
Set
\[ \mathscr P_{\rm sm}(\mathbf x) =\prod_{m\ge1}\frac{1}{1-\LL^{m-1}\mathbf x^{mu}},
 \qquad \mathscr P_a(\mathbf x)=\mathscr P_{\rm sm}(\mathbf x) \prod_{j=1}^{r_a-1}\prod_{m\ge0} \frac{1}{1-\LL^m\mathbf x^{mu+\beta_{a,j}}},
\]
and
\[ \mathscr P_p(\mathbf x)= \sum_{\mathbf v}[P_{\mathbf v}(G_p)] \mathbf x^{\iota_p(\mathbf v)}.
\]
The geometric power structure gives
\begin{equation}
\label{eq:motivic-alpha}
\sum_\alpha[\Quot^\alpha\mcX]\mathbf x^\alpha =\mathscr P_{\rm sm}(\mathbf x)^{[U]} \prod_a\mathscr P_a(\mathbf x)^{[D_a^\circ]} \prod_p\mathscr P_p(\mathbf x).
\end{equation}
This is understood after the local representation classes are mapped to $\Lambda_{\mcX}$, so induction and incidence relations are imposed before a coefficient is extracted.
It follows by splitting disjoint supports and trivializing the bounded punctual jet families on a finite stratification,
as in \cite{GuseinZadeLuengoMelle2010}.

Applying $E_c$ to \eqref{eq:motivic-alpha} is always legitimate.
If the coefficient Quot spaces are smooth projective, evaluating $E_c$ at $u=v=-q$ recovers their ordinary Betti polynomials.
In general, \eqref{eq:IC-Betti} is the ordinary-cohomology statement.
Its factors use closed support loci, not ordinary Betti polynomials of the open strata.
In particular the moving surface factor below involves the coarse surface $X$, not just $U$.

\section{Support decomposition and compact intersection forms}
\label{sec:support}
\label{qfig:sec-intersection}

\subsection{The general support statement}
\label{subsec:general}

Let $\Lambda_{\mcX}$ be the numerical class group from Section \ref{subsec:motivic-assembly}, and put $u([\mcE])=\operatorname{rk}\mcE$.
Write $\mathbf x^\alpha$ for the monomial recording a class $\alpha\in\Lambda_{\mcX}$.
All sums are taken in a completion in which a bounded positive length has finitely many terms.
Fix a smooth component $H\subset\Quot^\alpha\mcX$, put $d=\dim H$,
and suppose that its HC morphism $\pi=h_\alpha|_{H}:H\ra C_\alpha$ is projective.
The cycle space may retain the several character cycles in a root chart; forgetting them is not needed here.

\begin{theorem}[Support decomposition]
\label{thm:general}
For a support stratification adapted to $\pi$, there is a noncanonical decomposition
\begin{equation*}
 R\pi_*\QQ_H[d] \simeq\bigoplus_{Z,k}\operatorname{IC}_Z(\mcL_{Z,k})[-k],
\end{equation*}
where $Z$ runs over a collision stratification refined by stabilizer and local representation data and adapted to \(\pi\), 
and $\mcL_{Z,k}$ are semisimple coefficient local systems, including possibly zero multiplicities.
With perverse normalization of the intersection complex,
\begin{equation}
\label{eq:IC-Betti}
P(H;q)=\sum_{Z,k}q^{d+k-\dim Z} P\bigl(IH^*(Z,\mcL_{Z,k});q\bigr).
\end{equation}
The transverse calculations at distinct support points are products of punctual calculations.
Their monodromy is retained in $\mcL_{Z,k}$.
\end{theorem}
\begin{proof}
Apply the BBDG's decomposition theorem \cite{BBD} to the smooth source and projective map.
Our convention is $IH^a(Z,\mcL)=\mathbb H^{a-\dim Z}(Z,\operatorname{IC}_Z(\mcL))$;
taking hypercohomology gives the exponent in \eqref{eq:IC-Betti}.
On a neighborhood of a cycle with disjoint support, the Quot functor is the product of its support-local functors.
The corresponding proper direct images are exterior products, which proves the last assertion.
\end{proof}

\subsection{Along a pure stacky curve}
\label{subsec:curve}

Let $D\subset S$ be a smooth divisor in a smooth quasiprojective surface, and let $H=H^{(r)}_{k,\boldsymbol\eta}(S,D)$ be the root-ideal flag space of Section \ref{sec:punctual-reduction}.
Here $\eta_j=\ell(J_j/J_{j-1})\ge0$ and $k=\ell(\mcO_S/J_{r-1})$.
The nonnegative gaps are distinct from the signed character defects:
\[ v_i=k+\sum_{j>i}\eta_j,
 \qquad v_i-v_0=-\sum_{j\le i}\eta_j.
\]
Thus the invariant length is $v_0=k+|\boldsymbol\eta|$, whereas $k$ is the colength of the largest ideal in the flag.
Its support morphism is
\begin{equation}
\label{eq:root-map}
\pi_{k,\boldsymbol\eta}:H\longrightarrow \Sym^kS\times\prod_{j=1}^{r-1}\Sym^{\eta_j}D,
 \quad (J_\bullet)\longmapsto ([\mcO_S/J_{r-1}],([J_j/J_{j-1}])_j).
\end{equation}
It is projective, and $H$ is smooth of dimension $d=2k+|\boldsymbol\eta|$.

\begin{prop}[Curve supports]
\label{prop:curve-supports}
The map \eqref{eq:root-map} is semismall.
Its relevant support types are sums of the modes
\[ (m,\mathbf0),\quad m\ge1,
 \qquad (m,\mathbf e_j),\quad m\ge0,
\]
moving respectively on $S$ and $D$.
Every relevant coefficient local system has rank one and is constant.
\end{prop}
\begin{proof}
At one point a block of type $(m,\mathbf0)$ has punctual dimension $m-1$, whereas a block $(m,\boldsymbol\epsilon)$ with $|\boldsymbol\epsilon|>0$ has punctual dimension $m$.
These dimensions, and purity of the dimensions of the fibres,
are given by the parabolic punctual paving.
The ordinary block moves in dimension two and contributes codimension $2m-2$.
A defect block moves on $D$, so its contribution to codimension is $2m+|\boldsymbol\epsilon|-1\ge2m$.
Equality holds precisely when $|\boldsymbol\epsilon|=1$.
Add these inequalities over distinct support points.
Restricting an ordinary block to $D$ adds one to codimension without increasing the fibre dimension and creates no additional relevant support.

For a relevant ordinary block, Brian\c{c}on's theorem gives the unique top component.
For a relevant defect block the top-cell coefficient is $[t^m]\frac{1}{1-t}=1$, and the parabolic paving has dense top-dimensional cells.
Thus again there is a unique top component.
Complex coordinate changes preserve its fundamental class, which proves triviality of the local systems.
This is the semismall calculation of \cite[Appendix, Section~8]{deCataldoMiglioriniSemismall2004}.
\end{proof}

More explicitly, choose finite collections of nonnegative integers $a_m,b_{j,m}$ such that
\[ k=\sum_{m\ge1}ma_m+\sum_{j,m\ge0}mb_{j,m},
 \qquad \eta_j=\sum_{m\ge0}b_{j,m}.
\]
The normalization of the corresponding closed support is
\[ Z_{a,b}=\prod_{m\ge1}\Sym^{a_m}S \times\prod_{j,m\ge0}\Sym^{b_{j,m}}D.
\]
Let $\nu_{a,b}$ be its finite map to the target of \eqref{eq:root-map}.
The semismall decomposition reads
\begin{equation}
\label{eq:root-IC}
R(\pi_{k,\boldsymbol\eta})_*\QQ_H[d] \simeq\bigoplus_{a,b} (\nu_{a,b})_*\QQ_{Z_{a,b}}[\dim Z_{a,b}].
\end{equation}
Indeed symmetric powers of a smooth surface are rational homology manifolds, and symmetric powers of a smooth curve are smooth.
Consequently
\begin{equation*}
 P(H;q)=\sum_{a,b}q^{2c(a,b)}P(Z_{a,b};q),
 \qquad c(a,b)=\sum_{m\ge1}(m-1)a_m+\sum_{j,m\ge0}mb_{j,m}.
\end{equation*}
Taking graded symmetric powers gives the product of Section \ref{sec:punctual-reduction}:
\begin{equation}
\label{eq:root-Betti}
\sum_{k,\boldsymbol\eta}P(H^{(r)}_{k,\boldsymbol\eta};q) t^k\mathbf z^{\boldsymbol\eta} =F_S(q,t) \prod_{j=1}^{r-1}\prod_{m\ge0} \frac{(1+q^{2m+1}t^mz_j)^{b_1(D)}} {(1-q^{2m}t^mz_j)^{b_0(D)} (1-q^{2m+2}t^mz_j)^{b_2(D)}}.
\end{equation}
The support proof establishes this ordinary Betti identity also for quasiprojective $S$ and $D$; it does not require projectivity or purity of their cohomology.

\begin{rmk}
For $|\boldsymbol\eta|>1$, a fibre over one point of $D$ can have several top components, but its stratum is not relevant.
Those classes are accounted for by stalks of the larger relevant supports in \eqref{eq:root-IC}.
The curve coefficient is a transverse intersection form after slicing the moving curve direction; it is not the compact-to-ordinary image of the whole affine chart.
For example $H^{(r)}_{0,\mathbf e_j}(\CC^2,\CC)=\CC$ has a one-point punctual fibre and a rank-one curve support,
but $\operatorname{im}(H_c^*(\CC)\ra H^*(\CC))=0$.
At an isolated small-group point the relevant support really can be a point, and it is precisely there that the compact images $I_{\mathbf v}(G)$ and the series $K_G$ enter.
\end{rmk}

\subsection{Application to isolated stacky points}
Locally at isolated stacky points, we can put $X=\CC^2/G$ for small finite group $G$, and let $[0]\in X$ be the image of the origin.
For a charge $\mathbf v$, set
\[ a(\mathbf v)=\min_i\left\lfloor\frac{v_i}{d_i}\right\rfloor.
\]

\begin{prop}
\label{prop:support-map}
There is a projective morphism
\[ \pi_{\mathbf v}:\Quot^{\mathbf v}[\CC^2/G]\longrightarrow \Sym^{a(\mathbf v)}X.
\]
If a quotient has moving part supported on free orbits with total balanced multiplicity $m$, this map records their length-$m$ cycle on $X\setminus\{[0]\}$ and adds $(a(\mathbf v)-m)[0]$.
In particular,
\begin{equation*}
\pi_n:\Hilb^n[\CC^2/G]\longrightarrow\Sym^nX,\qquad P_n=\pi_n^{-1}(n[0]).
\end{equation*}
Write $P_{\mathbf v}$ for the origin fibre in general.
\end{prop}

\begin{proof}
Restrict the ordinary Hilbert-Chow morphism in colength $N=|\mathbf v|$.
By smallness every nonzero orbit is free.
A $G$-invariant quotient supported on such an orbit is induced from its quotient at one orbit point,
and hence has representation $k\CC[G]$ for some $k$.
Its moving part therefore has charge $m\delta$, where $0\le m\le a(\mathbf v)$.

The orbit-cycle map
\[ \iota_{\mathbf v}:\Sym^{a(\mathbf v)}X \longrightarrow\Sym^N\CC^2 \]
replaces each point by its orbit cycle, counted with multiplicity,
and adds $(N-a(\mathbf v)|G|)0$.
An orbit cycle over $o$ is $|G|\cdot0$.
This map is a closed embedding.
To see this on coordinate rings, take an invariant polynomial $f$ with $f(0)=0$.
The ambient power sum of $f$ pulls back to $|G|$ times the power sum of the descended function on $X$.
Such power sums generate the coordinate ring of a symmetric product in characteristic zero,
so the pullback is surjective.

The restricted Hilbert-Chow map has image in this closed subscheme on geometric points.
Since $\Quot^{\mathbf v}[\CC^2/G]$ is reduced, it factors scheme-theoretically through $\iota_{\mathbf v}$.
Its factor is projective, being the base change of the restricted projective map along a closed embedding.
\end{proof}

\begin{lem}
\label{lem:core}
Scalar dilation makes $\Quot^{\mathbf v}[\CC^2/G]$ semi-projective.
Its proper core is $P_{\mathbf v}$, and the inclusion induces isomorphisms on ordinary rational homology and cohomology.
The group $H^d(P_{\mathbf v})$, equivalently $H^d(\Quot^{\mathbf v}[\CC^2/G])$, is pure of weight $d$.
\end{lem}

\begin{proof}
The affine base $\Sym^{a(\mathbf v)}X$ has a positive grading, with unique fixed point $a(\mathbf v)[0]$.
Projectivity of $\pi_{\mathbf v}$ extends every scalar orbit as its parameter tends to zero.
Fixed ideals are homogeneous and punctual, so the fixed locus is a closed subvariety of the projective fibre $P_{\mathbf v}$.

The core consists of points whose scalar orbit also has a limit as the parameter tends to infinity.
Such a limit can exist on the positively graded base only at its origin.
Conversely every orbit in the proper fibre $P_{\mathbf v}$ has that limit by properness.
Thus the core is exactly $P_{\mathbf v}$.
The core theorem for smooth semiprojective varieties gives the homology and cohomology isomorphisms \cite[Theorem 1.3.1 and Corollaries 1.3.2, 1.3.6]{HRV}.

The inclusion is algebraic, so its cohomology isomorphism preserves mixed Hodge structures.
Smoothness bounds the weights of $H^d(\Quot^{\mathbf v}[\CC^2/G])$ below by $d$, whereas properness bounds those of $H^d(P_{\mathbf v})$ above by $d$.
The common group is therefore pure.
\end{proof}

Using the notation of Section \ref{sec:preliminaries},
\begin{equation*}
B_G(q,t)=\sum_{n\ge0}P(\Hilb^n[\CC^2/G];q)t^n =\sum_{n\ge0}P(P_n;q)t^n.
\end{equation*}
Purity does not by itself assert that all these cohomology groups are even or of Tate type.

\subsection{The local collision space}

\begin{defn}
\label{def:collision}
The compact intersection image is
\[ I_{\mathbf v}^d =\operatorname{im}\bigl(H_c^d(\Quot^{\mathbf v}[\CC^2/G]) \longrightarrow H^d(\Quot^{\mathbf v}[\CC^2/G])\bigr).
\]
Set $I_n=I_{n\delta}$ and
\begin{equation*}
K_G(q,t)=\sum_{n,d\ge0}\dim I_n^d\,q^dt^n.
\end{equation*}
Empty moduli spaces contribute zero; $\Hilb^0[\CC^2/G]$ is a point and $I_{\mathbf0}=\QQ$.
\end{defn}

The full punctual fibre and this image play different roles.
Lemma \ref{lem:core} identifies $H^*(P_n)$ with all of $H^*(\Hilb^n[\CC^2/G])$.
The image $I_n$ will instead be the multiplicity attached to support concentrated at the origin.

\subsection{Support strata and their intersection complexes}

For $0\le m\le n$, let $S_{m,\lambda}\subset\Sym^aX$ consist of cycles
\begin{equation*}
(n-m)[0]+\lambda_1x_1+\cdots+\lambda_\ell x_\ell,
\quad x_i\ne o, x_i\ne x_j\quad(i\ne j).
\end{equation*}
The dimension vector at the origin of a quotient over this stratum is $\bfw=\mathbf v-m\delta$.

\begin{lem}
\label{lem:normalization}
The normalization of $\overline{S}_{m,\lambda}$ is
\[ Z_\lambda=\prod_{k\ge1}\Sym^{a_k(\lambda)}X.
\]
Its finite normalization map is $\nu_\lambda((z_k)_k)=(a-m)[0]+\sum_k k z_k$.
The variety $Z_\lambda$ is contractible and rationally smooth of complex dimension $2\ell$, and
\begin{equation}
\label{eq:ic-normalization}
\operatorname{IC}_{\overline S_{m,\lambda}} \cong R\nu_{\lambda*}\QQ_{Z_\lambda}[2\ell].
\end{equation}
The intersection complex is normalized to be perverse.
\end{lem}

\begin{proof}
Standard treatment of symmetric product, a simple modification from \cite{De_Cataldo2002-xj}.
\end{proof}

For $\ell>0$, this lemma and rational Poincar\'e duality give
\[ H^i(Z_\lambda)=
\begin{cases}
\QQ&i=0,\\0&i\ne0,
\end{cases}
\qquad H_c^i(Z_\lambda)=
\begin{cases}
\QQ&i=4\ell,\\0&i\ne4\ell.
\end{cases}
\]
Consequently the natural compact-to-ordinary map for the complex in \eqref{eq:ic-normalization} is zero.
For the empty partition,
$Z_\varnothing$ is a point and that map is the identity.

We then compute their local multiplicites of each perverse summand.
\begin{lem}
\label{lem:local-multiplicities}
Let
\[ \mathcal C_{\mathbf v} =R\pi_{\mathbf v*}\QQ_{\Quot^{\mathbf v}[\CC^2/G]}[D(\mathbf v)].
\]
In the semisimple perverse sheaf ${}^pH^j(\mathcal C_{\bfw})$, let $V_{\bfw,j}$ denote the multiplicity of the skyscraper sheaf at the origin.
Then the summands of $\mathcal C_{\mathbf v}$ whose support is exactly $\overline S_{m,\lambda}$ are
\begin{equation}
\label{eq:exact-support}
\operatorname{IC}_{\overline S_{m,\lambda}} \otimes V_{\mathbf v-m\delta,j}[-j].
\end{equation}
All the coefficient local systems on $S_{m,\lambda}$ in this description are constant.
\end{lem}

\begin{proof}
Choose pairwise disjoint analytic neighborhoods of $o,x_1,\ldots,x_\ell$.
A finite subscheme supported in their union decomposes uniquely into its restrictions to the neighborhoods.
Above a nonzero point the action is free, so an equivariant quotient on the orbit is equivalent to an ordinary quotient on the corresponding neighborhood of the smooth surface $X\setminus\{o\}$.
The Hilbert-Chow morphism is therefore locally the product of the origin problem for $\bfw$ and the ordinary Hilbert-Chow maps in lengths $\lambda_1,\ldots,\lambda_\ell$.

For a smooth surface, Hilbert-Chow is semismall.
Its decomposition has one rank-one summand for each partition stratum: the relevant punctual Hilbert fibre has a unique top-dimensional fundamental class \cite{De_Cataldo2002-xj}.
Transversely to $S_{m,\lambda}$, the summand with that exact support uses this top class in every moving factor.
In the origin factor it must use a point-supported summand.
A positive-dimensional support in that factor instead gives a strictly larger support closure.
The product decomposition proves \eqref{eq:exact-support}.
There is no additional perverse shift: the smooth source dimensions satisfy $D(\mathbf v)=D(\bfw)+2m$.

This product description also makes the direct image constructible for the stated support stratification; its simple supports are closures of these strata.
The moving top classes have trivial monodromy under changes of analytic coordinates, since such changes preserve complex orientation.
Permuting equal-sized moving factors acts trivially as well: their fundamental classes have even real degree.
The origin factor is fixed independently of the moving configuration.
Thus no nonconstant local system is introduced.
\end{proof}

\subsection{The local factorization}

We know have the following factorization via decomposition of support.
\begin{theorem}
\label{thm:support}
For every finite small subgroup $G\subset GL_2(\CC)$,
\begin{equation}
\label{eq:factorization}
B_G(q,t)=F(q,t)K_G(q,t),\qquad F(q,t)=\prod_{k\ge1}\frac{1}{1-q^{2k-2}t^k}.
\end{equation}
More generally, with $\mathbf z^{\mathbf v}=\prod_i z_i^{v_i}$,
\begin{equation}
\label{eq:all-charge}
\sum_{\mathbf v}P(\Quot^{\mathbf v}[\CC^2/G];q)\mathbf z^{\mathbf v} = \prod_{k\ge1}\frac{1}{1-q^{2k-2}\mathbf z^{k\delta}} \sum_{\mathbf v}P(I_{\mathbf v};q)\mathbf z^{\mathbf v}.
\end{equation}
Moreover, the multiplicity of $IC_{\bar{S}_{0,\lambda}}$ in degree $j$ is isomorphic to $I_{\mathbf v}^{D(\mathbf v)+j}$.
\end{theorem}

\begin{proof}
Since $\pi_{\mathbf v}$ is projective and $\Quot^{\mathbf v}[\CC^2/G]$ is smooth, the decomposition theorem applies to $\mathcal C_{\mathbf v}$ \cite{BBD,deCataldoMigliorini}.
Choose a decomposition into shifted semisimple perverse sheaves,
and use Lemma \ref{lem:local-multiplicities} to describe its simple summands.
Apply the natural map from compactly supported to ordinary hypercohomology to this same decomposition.

By Lemma \ref{lem:normalization}, every positive-dimensional support contributes zero to this map.
A point-supported summand $V_{\mathbf v,j}[-j]$ contributes its whole multiplicity, in ordinary degree $D(\mathbf v)+j$.
Hence
\begin{equation}
\label{eq:point-image}
V_{\mathbf v,j}\cong I_{\mathbf v}^{D(\mathbf v)+j}.
\end{equation}
The splitting of the decomposition theorem need not be canonical.
Its multiplicities, and the compact intersection image used here,
are independent of that choice.

Now consider the summand indexed by $(m,\lambda,j)$, with $\bfw=\mathbf v-m\delta$.
Its normalization has ordinary cohomology only in degree zero.
Including the intersection-complex and perverse shifts, its contribution to $H^d(\Quot^{\mathbf v}[\CC^2/G])$ occurs at
\[ d=D(\mathbf v)+j-2\ell(\lambda) =D(\bfw)+j+2m-2\ell(\lambda).
\]
Together with \eqref{eq:point-image}, this proves the coefficient identity
\begin{equation}
\label{eq:coefficient}
P(\Quot^{\mathbf v}[\CC^2/G];q) = \sum_{m=0}^{a(\mathbf v)} \ \sum_{\lambda\vdash m} q^{2m-2\ell(\lambda)} P(I_{\mathbf v-m\delta};q).
\end{equation}
All sums in a fixed charge are finite.

A part of size $k$ contributes degree $2k-2$ and charge $k\delta$, and it may occur any number of times.
Therefore
\[ \sum_{m\ge0}\sum_{\lambda\vdash m} q^{2m-2\ell(\lambda)}\mathbf z^{m\delta} = \prod_{k\ge1}\frac{1}{1-q^{2k-2}\mathbf z^{k\delta}}.
\]
Summing \eqref{eq:coefficient} gives \eqref{eq:all-charge}.
On the balanced ray $\mathbf v=n\delta$, every residual charge is $(n-m)\delta$; this gives \eqref{eq:factorization}.
\end{proof}

\begin{corollary}[Reconstruction from collision cohomology]
\label{cor:reconstruction}
If $K_n(q)=\sum_d\dim I_n^d q^d$, then
\[ P(P_n;q)=P(\Hilb^n[\CC^2/G];q) =\sum_{m=0}^n\ \sum_{\lambda\vdash m} q^{2m-2\ell(\lambda)}K_{n-m}(q).
\]
Thus the graded dimensions of the geometrically defined spaces $I_n$ determine the full Betti series.
\end{corollary}

\begin{proof}
Combine \eqref{eq:coefficient} with Lemma \ref{lem:core}.
\end{proof}

\subsection{The Betti product on a general surface stack}

The local support theorem globalizes in the locally balanced sector.
For a quotient surface $X$, define its coarse surface factor by
\begin{equation}
\label{eq:coarse-factor}
\mathcal F_X(q,t)=\prod_{k\ge1} \frac{\prod_{i\text{ odd}}(1+q^{i+2k-2}t^k)^{b_i(X)}} {\prod_{i\text{ even}}(1-q^{i+2k-2}t^k)^{b_i(X)}}.
\end{equation}
Only $0\le i\le4$ occur.
If $X$ is smooth this is $F_X$.
For singular $X$ the definition uses the cohomology of $X$; it is NOT the Betti series of its possibly singular Hilbert schemes.

\begin{theorem}[Global locally balanced factorization]
\label{thm:global-betti}
Let $\mcX$ be an effective smooth surface stack with quasi-projective coarse space $X$, and assume its locally balanced Hilbert spaces are of finite type with projective Hilbert-Chow morphisms.
For $p\in\operatorname{Sing}X$, let $G_p^{\mathrm{can}}$ be the small tangent stabilizer of $X^{\mathrm{can}}$.
Then
\begin{equation*}
 \boxed{\quad B_{\mcX}^{\mathrm{lr}}(q,t):= \sum_{n\ge0}P(\Hilb^n_{\mathrm{lr}}(\mcX);q)t^n =\mathcal F_X(q,t) \prod_{p\in\operatorname{Sing}X}K_{G_p^{\mathrm{can}}}(q,t).
 \quad}
\end{equation*}
This is an identity of ordinary Betti series, also when $X$ is nonproper.
It does not require purity of its cohomology.
\end{theorem}
\begin{proof}
By \eqref{eq:canonical-reduction}, replace $\mcX$ by $X^{\mathrm{can}}$.
Its stacky locus is finite; write it as $P$, with groups $G_p$.
Local balance is an open-and-closed condition here.
In a neighborhood of $p$, every moving orbit contributes a regular representation, so the residual charge modulo the regular representation is locally constant.
The local dimension calculation gives $\dim\Hilb^n_{\mathrm{lr}}=2n$.
The Hilbert-Chow morphism records the integer regular lengths in $\Sym^nX$.

For $n_p\ge0$ and $\lambda\vdash m=n-\sum_pn_p$, consider the stratum of cycles
\[ \sum_{p\in P}n_pp+\sum_{a=1}^{\ell(\lambda)}\lambda_ax_a,
 \qquad x_a\in X\setminus P,\quad x_a\ne x_b\ (a\ne b).
\]
The normalization of its closure is $Z_\lambda=\prod_{k\ge1}\Sym^{a_k(\lambda)}X$, with the finite map which weights each cycle by $k$ and adds $\sum_pn_pp$.
It is generically one-to-one.
Quotient singularities, their products and symmetric powers are rational homology manifolds over $\QQ$,
by the finite-quotient local homology argument of Lemma \ref{lem:normalization}.
Consequently the intersection complex is the pushforward of $\QQ_{Z_\lambda}[2\ell(\lambda)]$.

The disjoint-support argument of Lemma \ref{lem:local-multiplicities} applies at every $p$.
An exact-support summand uses the unique ordinary punctual top class at each moving point and a point-supported local multiplicity $V_{p,n_p,j_p}=I_{n_p}^{2n_p+j_p}(G_p)$ at each fixed point.
The moving fundamental classes have even degree and are preserved by changes of complex coordinates and by permutations.
The fixed point factors are independent of the moving configuration.
Thus their coefficient local systems are constant.
The dimension identity $2n=2m+\sum_p2n_p$ gives the contribution
\[ q^{2m-2\ell(\lambda)}P(Z_\lambda;q) \prod_{p\in P}P(I_{n_p}(G_p);q).
\]
This calculation uses the decomposition theorem, not additivity over locally closed strata.

Finally $H^*(\Sym^aX;\QQ)$ is the invariant part of $H^*(X^a;\QQ)$, hence the graded-symmetric power of $H^*(X;\QQ)$.
An even generator gives a denominator and an odd generator an exterior numerator.
Summing over $a_k(\lambda)$ therefore gives
\[ \sum_{m\ge0}\sum_{\lambda\vdash m} q^{2m-2\ell(\lambda)}P(Z_\lambda;q)t^m=\mathcal F_X(q,t).
\]
Summing independently over the $n_p$ proves the theorem.
\end{proof}

For $X=\CC^2/G$ with $G$ small, contraction gives only $b_0(X)=1$.
The product is exactly $B_G=F K_G$.
For a root stack over a smooth surface, the coarse singular set is empty and the locally regular product is the ordinary surface series.
The full root-stack series still contains the curve factors of \eqref{eq:root-Betti}.

The same proof gives a full, all-charge extension when the effective stack has only isolated stacky points.
Set
\[ K^{\mathrm{all}}_{G_p}(q,T) =\sum_{\mathbf v}P(I_{\mathbf v}(G_p);q) T^{|\mathbf v|/|G_p|}.
\]
If specialization of the numerical classes to stack degree has finite coefficient sums, and the corresponding Hilbert-Chow morphisms are projective,
then
\begin{equation*}
 \sum_dP(\Quot_d^{\mathrm{all}}(\mcX);q)T^d =\mathcal F_X(q,T)\prod_pK^{\mathrm{all}}_{G_p}(q,T).
\end{equation*}
Indeed, a stratum with residual charges $\mathbf v_p$ has source dimension $2m+\sum_pD(\mathbf v_p)$.
Substituting the local multiplicity $I_{\mathbf v_p}^{D(\mathbf v_p)+j_p}$ again leaves the shift $2m-2\ell(\lambda)$.
One sums separately over smooth components of each dimension.
A common multiple of the stabilizer orders clears the denominators of the support cycles.

The factors $K_{G_p}$ are compact intersection images in the \emph{affine local models}.
They are not the compact-to-ordinary image of the global Quot space, which for a projective space is its entire cohomology.
Thus the global theorem leaves a definite local geometric task: understand the formation of punctual algebras and their ambient intersection maps.
Sections \ref{qfig:sec-intersection} and \ref{sec:formation} undertake that task.
In the full theory with stacky curves, the same task includes the coupled crossings of the coupled crossing grid in the sequel work; its general assembly is \eqref{eq:motivic-alpha} and \eqref{eq:IC-Betti}.

\subsection{Lefschetz symmetry and compact intersections}

\begin{theorem}[Lefschetz symmetry of the collision spaces]
\label{thm:lefschetz}
Let $L_{\mathbf v}$ be cup product with the first Chern class of a $\pi_{\mathbf v}$-ample line bundle.
It preserves $I_{\mathbf v}$, and for every $j\ge0$,
\begin{equation}
\label{eq:lefschetz}
L_{\mathbf v}^{\,j}:
I_{\mathbf v}^{D(\mathbf v)-j} \xrightarrow{\ \sim\ } I_{\mathbf v}^{D(\mathbf v)+j}.
\end{equation}
In particular,
\[ P(I_n;q)=q^{4n}P(I_n;q^{-1}).
\]
\end{theorem}

\begin{proof}
The map $H_c^*(\Quot^{\mathbf v}[\CC^2/G])\ra H^*(\Quot^{\mathbf v}[\CC^2/G])$ is compatible with cup product by ordinary cohomology classes.
Its image is therefore preserved by $L_{\mathbf v}$.

Relative hard Lefschetz for the projective map $\pi_{\mathbf v}$ gives isomorphisms
\[ L_{\mathbf v}^{\,j}:
{}^pH^{-j}(\mathcal C_{\mathbf v}) \xrightarrow{\ \sim\ } {}^pH^j(\mathcal C_{\mathbf v}).
\]
Taking point-supported multiplicities gives $V_{\mathbf v,-j}\cong V_{\mathbf v,j}$.
Under \eqref{eq:point-image} these are the groups in \eqref{eq:lefschetz}.
This identifies the actual cup-product map:
in cohomological degree $d$, the induced perverse filtration on the compact intersection image has only the single possible degree $d-D(\mathbf v)$.
There are thus no other perverse graded pieces on that image into which its cup-product operator can mix.
Finally $D(n\delta)=2n$, giving the displayed symmetry.
\end{proof}

Equivalently, $I_{\mathbf v}$ is a finite-dimensional Lefschetz representation centered at $D(\mathbf v)$.
If
\[ \operatorname{Prim}^{D-a} =\ker\bigl(L_{\mathbf v}^{a+1}:
I_{\mathbf v}^{D-a}\ra I_{\mathbf v}^{D+a+2}\bigr),
\qquad D=D(\mathbf v),
\]
then it is the direct sum of the strings
\[ \operatorname{Prim}^{D-a},\ L_{\mathbf v}\operatorname{Prim}^{D-a},\ \ldots,\ L_{\mathbf v}^{a}\operatorname{Prim}^{D-a}.
\]
The collision space contains these entire strings.
It is not only their lowest-degree primitive subspace.

\begin{prop}[Intersection-image interpretation]
\label{prop:intersection-image}
Let $i:P_{\mathbf v}\hookrightarrow \Quot^{\mathbf v}[\CC^2/G]$ and $D=D(\mathbf v)$.
The dimension of $I_{\mathbf v}^d$ is the rank of the map
\begin{equation}
\label{eq:intersection-map}
H_{2D-d}(P_{\mathbf v}) \xrightarrow{i_*}H_{2D-d}^{\mathrm{BM}}(\Quot^{\mathbf v}[\CC^2/G]) \xrightarrow{\mathrm{PD}}H^d(\Quot^{\mathbf v}[\CC^2/G]) \xrightarrow{i^*}H^d(P_{\mathbf v}).
\end{equation}
Equivalently, it is the rank of the pairing of this map with $H_d(P_{\mathbf v})$.
\end{prop}

\begin{proof}
Poincar\'e duality on the smooth variety $\Quot^{\mathbf v}[\CC^2/G]$ identifies the compact-to-ordinary map with the natural map
\[ H_{2D-d}(\Quot^{\mathbf v}[\CC^2/G]) \longrightarrow H_{2D-d}^{\mathrm{BM}}(\Quot^{\mathbf v}[\CC^2/G]).
\]
Lemma \ref{lem:core} identifies its ordinary homology source with $H_{2D-d}(P_{\mathbf v})$ and its cohomological target with $H^d(P_{\mathbf v})$.
This gives \eqref{eq:intersection-map}.
The rational homology-cohomology pairing on the proper space $P_{\mathbf v}$ is nondegenerate as a vector-space pairing,
so evaluating on $H_d(P_{\mathbf v})$ does not change the rank.
\end{proof}

For algebraic cycles of complementary dimensions the last pairing is their intersection in the smooth ambient $\Quot^{\mathbf v}[\CC^2/G]$.
This uses its normal embedding.
It is not a replacement of \eqref{eq:intersection-map} by a Poincar\'e pairing on the possibly singular space $P_{\mathbf v}$.

\subsection{Numerical classes and the global product without crossings}
\label{subsec:global-alpha}

Suppose the effective stack has disjoint smooth stacky curves $D_a$ with constant orders $r_a$, and isolated stacky points $p$ away from them.
Along the curves it has the root models of Section \ref{sec:support}; the $D_a$ lie in the smooth coarse locus.
Assume quasiprojective coarse space, finite-type coefficient Quot spaces, and projective support morphisms.
Let $G_p$ be the small tangent stabilizer at $p$ and define
\[ K_p(q,\mathbf x)=\sum_{\mathbf v} P(I_{\mathbf v}(G_p);q)\mathbf x^{\iota_p(\mathbf v)}.
\]
Use the coarse factor $\mathcal F_X$ of \eqref{eq:coarse-factor}.

\begin{theorem}[Global class-valued support product]
\label{thm:global-alpha}
Under these hypotheses, the ordinary Betti series is
\begin{align*}
\sum_\alpha P(\Quot^\alpha\mcX;q)\mathbf x^\alpha ={}&\mathcal F_X(q,\mathbf x^u) \\
&\cdot\prod_a\prod_{j=1}^{r_a-1}\prod_{m\ge0} \frac{(1+q^{2m+1}\mathbf x^{mu+\beta_{a,j}})^{b_1(D_a)}} {(1-q^{2m}\mathbf x^{mu+\beta_{a,j}})^{b_0(D_a)} (1-q^{2m+2}\mathbf x^{mu+\beta_{a,j}})^{b_2(D_a)}} \\
&\cdot \prod_p K_p(q,\mathbf x).
\end{align*}
For the numerical balanced series one extracts, from the entire right-hand side, the coefficients of $\mathbf x^{nu}$:
\begin{equation}
\label{eq:numerical-balanced}
B_{\mcX}^{\rm num-reg}(q,t) =\sum_{n\ge0}t^n[\mathbf x^{nu}] \left(\text{right-hand side of theorem \ref{thm:global-alpha}}\right).
\end{equation}
\end{theorem}
\begin{proof}
The normalization of a closed support is a product of symmetric powers of $X$ and the $D_a$, with fixed residual charges at the isolated points.
Its transverse calculation is the product of the ordinary punctual calculation, the relevant root calculations in Proposition \ref{prop:curve-supports},
and the point-supported multiplicities of the isolated local theorem.
The first two factors have constant rank-one coefficients.
The last coefficients are the fixed vector spaces $I_{\mathbf v}(G_p)$, with their cohomological degrees.
Thus coefficient monodromy is trivial.

An ordinary mode of length $m$ has degree shift $2m-2$; a curve mode of type $(m,\mathbf e_j)$ has shift $2m$.
A point factor contributes its already graded space $I_{\mathbf v}(G_p)$.
These shifts follow by subtracting the dimension of the support from that of the smooth source; the identity $D(\mathbf v+m\delta)=D(\mathbf v)+2m$ supplies the isolated calculation.
Quotient surface singularities and their symmetric powers are rational homology manifolds over $\QQ$,
so their intersection complexes are the shifted constant complexes.
Taking the graded symmetric powers of the moving supports proves theorem \ref{thm:global-alpha}.
Sum separately over source components if the fixed numerical class has components of different dimensions.
Finally \eqref{eq:numerical-balanced} is the definition of the global numerical class condition.
\end{proof}

This formulation does not replace the numerical condition by a pointwise balanced-representation condition.
The maps $\iota_p$ and the curve classes may satisfy relations.
An identification of the numerical balanced sector with the locally balanced sector therefore requires a detection statement for those relations.
For the separately defined locally balanced sector, all curve gaps vanish,
the isolated charges are $n_p\operatorname{Reg}(G_p)$, and the formula simplifies to
\[ B_{\mcX}^{\rm lr}(q,t)=\mathcal F_X(q,t)\prod_pK_{G_p}(q,t).
\]
At a crossing or stabilizer jump, a curve-mode support can acquire new point-supported multiplicities.
One must use the coupled local model and its local systems in Theorem \ref{thm:general}; the independent-curve product in \ref{thm:global-alpha} is not asserted there.
These cases will be treated in the sequel paper.

Let $M=\Quot^{\mathbf v}[\CC^2/G]$, $D=D(\mathbf v)$, and $P=P_{\mathbf v}$ in the next three subsections.
Write $I^d(M)=\operatorname{im}(H_c^d(M)\ra H^d(M))$.

\subsection{The normal bundle is part of the answer}

An additive invariant of the abstract fibre need not determine the intersection map.
For example, let
\[ Z_r=\operatorname{Tot}_{\mathbb P^1}\mcO_{\mathbb P^1}(-r),
 \qquad P=\mathbb P^1,
\]
where $P$ is the zero section.
These surfaces all retract onto the same proper variety, but
\begin{equation}
\label{qfig:eq-minus-r}
\langle[P],[P]\rangle =\int_P c_1(N_{P/Z_r})=-r.
\end{equation}
For $r=0$ its rank is zero; for $r>0$ its rank is one.
Even when the ranks agree, the integral matrices differ.
Thus neither the abstract fibre nor its class in a Grothendieck ring determines the intersection matrix.

For scalar $G=\mu_r$, $r\geq2$, this is precisely the balanced of length one moduli space.
Here is a direct construction.
The two charts for a regular quotient are
\begin{equation*}
 (x^r-s,\ y-tx),\qquad (y^r-s',\ x-t'y).
\end{equation*}
Their quotient bases are $1,x,\ldots,x^{r-1}$ and $1,y,\ldots,y^{r-1}$.
The transitions are $t'=t^{-1}$, $s'=t^r s$, giving \(\operatorname{Tot}\mcO(-r)\).
These charts cover: away from the origin a free scalar orbit has a nonzero coordinate, while at the origin the cotangent space cannot have dimension two,
since its scalar character has multiplicity one.
At least one coordinate therefore generates the punctual algebra.
The zero fibre is $s=0$, namely $\mathbb P^1$.
Consequently
\begin{equation}
\label{qfig:eq-length-one-K}
P(P_1;q)=1+q^2,\qquad K_1(q)=q^2.
\end{equation}
The matrix of the nonzero compact pairing is $(-r)$.

\section{Punctual geometry at isolated stacky points}
\label{sec:formation}

The support theorem reduces the unknown local character to the punctual algebra and its embedding in a smooth Quot space.
We first recall what can already be computed from torus fixed points, then explain which geometry the combinatorial count records. We warn in the beginning genuine difficulties arise with non-abelian finite groups, so most results in this section is about the cyclic non-crepant groups $C_{r,a}$.

\subsection{The cyclic calculation and balanced partitions}
\label{subsec:balanced-partitions}
For $G=\frac1r(1,a)$ with $(a,r)=1$, give the monomial $x^iy^j$ color $i+aj\pmod r$.
A monomial ideal of $R$-colength $nr$ has balanced charge precisely when its diagram contains $n$ boxes of each color.
Such a diagram is called balanced.
The commuting two-dimensional torus preserves the fixed Hilbert scheme, and its isolated fixed points are these monomial ideals.

Bejleri-Zaimi use a generic one-parameter subgroup and the resulting Bia\l ynicki-Birula cells to compute the Betti numbers \cite{Bejleri_Zaimi_2023}.
More precisely, choose a subgroup positive on both plane coordinates and generic for the fixed colength.
Let $a_+(\lambda)$ be the dimension of its positive tangent subspace in $\Hom_{R,G}(I_\lambda,R/I_\lambda)$.
Smoothness and semiprojectivity give attracting affine cells of these dimensions.
Consequently
\[ E_c(\Hilb^n[\CC^2/G];u,v) =\sum_{\lambda\ { m balanced}}(uv)^{a_+(\lambda)},\qquad P(\Hilb^n[\CC^2/G];q) =\sum_{\lambda\ { m balanced}}q^{4n-2a_+(\lambda)}.
\]
The second formula uses smooth duality, not additivity of ordinary cohomology over attracting cells.
Reversing the chamber and using the proper core gives the equivalent repelling-cell convention.
The tangent statistic is obtained from the two usual hook characters, retaining only the $G$-invariant characters.
The periodicity results of Bejleri-Zaimi concern variation of the cyclic group parameters.
Here $G$ is fixed and all lengths vary;
the support factorization extracts $K_G$ from these same Betti numbers and gives it the geometric meaning of an intersection image.
No cyclic partition description is asserted for arbitrary nonabelian small groups.

\subsection{Review of Brian\c{c}on's work}

Put $R=\CC[x,y]$ and $\mathfrak m=(x,y)$.
A punctual quotient is a finite local algebra $Q=R/I$ supported at the origin.
Its maximal-ideal filtration records the orders of vanishing of functions that remain after points collide.
Its socle,
\[ \operatorname{Soc}(Q)=\{a\in Q:\mathfrak m a=0\},
\]
records the last surviving functions.
We distinguish collisions of free orbits, deformations inside a punctual fibre, and deformations through curvilinear punctual algebras.

\begin{defn}
A punctual algebra of length $N\geq2$ is \emph{curvilinear} if its cotangent space has dimension one.
Equivalently it is isomorphic to $\CC[z]/(z^N)$.
\end{defn}

Indeed a lift of a cotangent generator generates the algebra, by Nakayama's lemma.
A finite quotient of $\CC[z]$ supported at zero then has the displayed form.

\begin{lem}
\label{qfig:lem-curvilinear}
On the ordinary punctual Hilbert scheme of length $N$, the chart where $x$ generates is
\begin{equation*}
 \left(x^N,\ y-\sum_{j=1}^{N-1}c_jx^j\right),
 \qquad (c_1,\ldots,c_{N-1})\in\mathbb A^{N-1}.
\end{equation*}
Together with the corresponding $y$-chart it covers the curvilinear locus.
The tangent-direction map makes this locus a Zariski locally trivial $\mathbb A^{N-2}$-fibration over $\mathbb P^1$.
\end{lem}

\begin{proof}
If $x$ generates, $Q=\CC[x]/(x^N)$, and the image of $y$ has a unique expansion without constant term.
Conversely the displayed family has free basis $1,x,\ldots,x^{N-1}$ over its entire parameter ring.
Either $x$ or $y$ has nonzero cotangent image.
In the first chart $c_1$ is the tangent slope, and $c_2,\ldots,c_{N-1}$ are the affine fibre coordinates.
The second chart gives the other standard chart of the space of tangent directions.
\end{proof}

Brian\c{c}on's theorem states that the whole punctual Hilbert scheme is irreducible of dimension $N-1$,
and that the curvilinear locus is dense \cite{Briancon1977}.
This is a deeper input than the preceding chart calculation.
One independent route uses Basili's theorem that
\[ \{(X,Y)\in\operatorname{End}(\CC^N)^2:
       [X,Y]=0,\ X,Y\text{ nilpotent}\} \]
is irreducible of dimension $N^2-1$ \cite{Basili}.
Adding a cyclic vector gives a nonempty open subset of its product with $\mathbb A^N$.
Sending a cyclic triple to $\ker(f\mapsto f(X,Y)v)$ is the $GL_N$ frame bundle of the tautological quotient.
The punctual Hilbert scheme is consequently irreducible of dimension $N^2-1+N-N^2=N-1$, and the nonempty open locus in Lemma \ref{qfig:lem-curvilinear} is dense.
The irreducibility assertion about commuting nilpotents is the substantive input in this argument.

\subsection{Obstruction of equivariant curvilinear smoothing}

\begin{prop}
\label{qfig:prop-equivariant-curvilinear}
Let $G\subset GL_2(\CC)$ be small.
If $G$ is nonabelian, it has no invariant curvilinear punctual quotient of length at least two.

For $G=\frac1r(1,a)$, $r\geq2$, $\gcd(a,r)=1$, and $n\geq1$,
the $x$-generated balanced curvilinear quotients are exactly
\begin{equation}
\label{qfig:eq-coloured-curvilinear}
\left(x^{nr},\ y-\sum_{\substack{1\leq j<nr\\j\equiv a\pmod r}}c_jx^j\right).
\end{equation}
They form $\mathbb A^n$.
If $a\ne1$, the full equivariant curvilinear locus is the disjoint union of the $x$- and $y$-generated charts.
If $a=1$, it is an $\mathbb A^{n-1}$-fibration over $\mathbb P^1$.
\end{prop}

\begin{proof}
A faithful finite nonabelian subgroup has irreducible representation $R_1$: otherwise complete reducibility diagonalizes it.
A curvilinear quotient would give a one-dimensional quotient of $R_1$,
namely its cotangent space, a contradiction.

For the cyclic group, uniqueness of the expression for $y$ in the basis $1,x,\ldots,x^{nr-1}$ imposes exactly the displayed character congruence.
There are $n$ allowed exponents, and the basis contains every character $n$ times.
In the $y$-chart the allowed exponents satisfy $aj\equiv1\pmod r$.
When $a\ne1$, neither expression has a linear term, so the other coordinate cannot generate; the charts are disjoint.
For $a=1$,
the allowable exponents are $1,r+1,\ldots,(n-1)r+1$, and the first coefficient is the tangent slope.
\end{proof}

These charts are actual free-orbit limits.
If $h(x)$ is the polynomial in \eqref{qfig:eq-coloured-curvilinear} and $b_1,\ldots,b_n$ are distinct nonzero constants,
then
\begin{equation*}
 \left(y-h(x),\ \prod_{j=1}^n(x^r-tb_j)\right)
\end{equation*}
is flat over $\CC[t]$ with basis $1,x,\ldots,x^{nr-1}$.
Its nonzero fibres are $n$ distinct free orbits; its zero fibre is the prescribed curvilinear quotient.

Nevertheless curvilinear quotients need not be dense equivariantly.

\begin{prop}
\label{qfig:prop-nondensity}
For scalar $\mu_r$, $r\geq2$, the balanced ideal $J=(x^2,y^r)$ lies outside the closure of the equivariant curvilinear locus in $P_2(\mu_r)$,
but is a flat limit of two free orbits.
\end{prop}

\begin{proof}
The basis
\[ 1,y,\ldots,y^{r-1},\quad x,xy,\ldots,xy^{r-1} \]
contains every scalar character twice.
In an equivariant curvilinear algebra of length $2r$, choose a linear generator $z$ and an independent linear form $\ell$ whose cotangent image vanishes.
The scalar character forces $\ell=c z^{r+1}$, whence $\ell R_{r-1}\subset I$.
The incidence
\[ \{(I,[\ell])\in P_2(\mu_r)\times\mathbb P(R_1):
                       \ell R_{r-1}\subset I\} \]
is closed.
Its projection is proper, so its image is closed and contains all the curvilinear quotients.

For $J$, writing $\ell=\alpha x+\beta y$, multiplication by $y^{r-1}$ and $xy^{r-2}$ gives respectively $\alpha xy^{r-1}$ and $\beta xy^{r-1}$.
The surviving basis vector $xy^{r-1}$ forces $\alpha=\beta=0$.
Thus $J$ is outside that closed image.
On the other hand,
\[ (x^2-ty^2,\ y^r-t) \]
has the same free basis over $\CC[t]$.
For $t\ne0$ its points satisfy $y^r=t$, $x=\pm\sqrt t\,y$, and form two free scalar orbits.
Its zero fibre is $J$.
\end{proof}

The ordinary punctual smoothing $(x^2-ty,y^r)$ also has this free basis and has curvilinear nonzero fibres.
It is not equivariant:
$x^2$ and $y$ have different scalar characters.
The preceding incidence argument excludes all equivariant curvilinear replacements,
not just this particular family.
The example includes the crepant case $r=2$.

\subsection{The complete scalar fibre in regular length two}

Let $R_j$ be the binary forms of degree $j$.
On $\mathbb P(R_2)\simeq\mathbb P^2$, define the rank-two bundle
\begin{equation*}
 0\longrightarrow R_{r-2}\otimes\mcO(-1) \xrightarrow{\ \cdot f\ }R_r\otimes\mcO \longrightarrow\mathcal E_r\longrightarrow0.
\end{equation*}
Multiplication by a nonzero quadratic is injective; the rank difference is two.
Put $D_r=\mathbb P_{\mathrm{quot}}(\mathcal E_r)$, the bundle of one-dimensional quotients.

\begin{theorem}
\label{qfig:thm-core-two}
For scalar $\mu_r$, $r\geq3$, the reduced punctual fibre contains the closed subvariety $D_r$, and
\begin{equation}
\label{qfig:eq-core-two}
P_2(\mu_r)_{\mathrm{red}}\setminus D_r \simeq \operatorname{Tot}\mcO_{\mathbb P^1}(r+2).
\end{equation}
It has an affine paving with dimensions
\[ 0,\quad1,1,1,\quad2,2,2,\quad3.
\]
Consequently
\begin{align}
 [P_2(\mu_r)_{\mathrm{red}}] &=1+3\Lef+3\Lef^2+\Lef^3,
\label{qfig:eq-motive-two}
\\ P(\Hilb^2[\CC^2/\mu_r];q) &=1+3q^2+3q^4+q^6,
\label{qfig:eq-B-two}
\\ K_2(q)&=q^2+3q^4+q^6.
\label{qfig:eq-K-two}
\end{align}
\end{theorem}

\begin{proof}
Let $Q=R/I$ be punctual and balanced of length $2r$.
Its embedding dimension is one or two.
If it is two, the first graded piece of its maximal-ideal filtration has dimension two and scalar character one.
Balance forces the piece in degree $r+1$ to vanish, hence every higher piece vanishes.
Each residue $1,\ldots,r-1$ can now occur only in its own degree and has dimension two.
The constant term uses one trivial character, so degree $r$ has dimension one.
Its Hilbert function is therefore
\[ (1,\underbrace{2,\ldots,2}_{r-1\text{ entries}},1).
\]
In particular $\mathfrak m^{r+1}\subset I$.
Modulo this ideal, distinct positive degrees below $r$ have distinct characters.
A relation of character zero has no constant term, because the quotient is supported at the origin.
Every relation is consequently homogeneous.

There is a unique quadratic line $\CC f\subset I_2$.
For $2\leq j<r$, its multiples $fR_{j-2}$ already have the required dimension $j-1$, so they are all relations in degree $j$.
The remaining datum is a line quotient $\phi:R_r/fR_{r-2}\twoheadrightarrow L$.
It defines the ideal
\begin{equation*}
 I_{f,\phi}=(f,\ \widetilde{\ker\phi}\subset R_r,\ \mathfrak m^{r+1}).
\end{equation*}
All graded pieces of its universal quotient are vector bundles of the indicated ranks.
This proves flatness over $D_r$.
Conversely the degree-two relation and the degree-$r$ quotient recover the parameters functorially on reduced bases: a universal multiplication map between vector bundles that vanishes on every geometric fibre vanishes on such a base.
Thus this locus is $D_r$.
It is closed: it is the locus of embedding dimension two, the complement of the open curvilinear locus,
and the displayed construction identifies its reduced structure with the projective bundle.

For embedding dimension one, a tangent chart has
\begin{equation}
\label{qfig:eq-two-curvilinear}
(y-tx-\kappa x^{r+1},\ x^{2r}).
\end{equation}
Changing to the other tangent chart gives
\[ t'=t^{-1},\qquad \kappa'=-t^{-r-2}\kappa.
\]
After a constant sign change on one chart this is the transition function of $\mcO(r+2)$, proving \eqref{qfig:eq-core-two}.
The closure of each affine fibre is also explicit.
Put $z=y-tx$.
The family
\begin{equation*}
 (z^2,\ x^{r-1}z,\ x^{r+1}-\eta z)
\end{equation*}
has free basis $1,x,\ldots,x^r,z,xz,\ldots,x^{r-2}z$.
For $\eta\ne0$ it is \eqref{qfig:eq-two-curvilinear} with $\kappa=\eta^{-1}$; its zero fibre belongs to $D_r$.
Hence each affine fibre attaches to a point of $D_r$.

The standard paving of $\mathbb P(R_2)$ has cells $\mathbb A^0,\mathbb A^1,\mathbb A^2$.
The bundle $\mathcal E_r$ is trivial on each: write the quadratic successively as $y^2$, $xy+c y^2$, or $x^2+bxy+c y^2$;
monic division gives two fixed remainder monomials in degree $r$.
The projective bundle therefore has cells of dimensions $0,1,1,2,2,3$.
Add to $D_r$ the affine fibre of \eqref{qfig:eq-core-two} over one point of $\mathbb P^1$.
This union is closed, since its complement is the restriction to the complementary affine line and is $\mathbb A^2$.
The two added cells have dimensions one and two.
This constructs the claimed affine paving.

The scissor relation gives \eqref{qfig:eq-motive-two}, also directly as $(1+\Lef)(1+\Lef+\Lef^2)+\Lef(1+\Lef)$.
Because $P_2$ is proper, the paving gives its ordinary Betti polynomial by replacing $\Lef$ with $q^2$.
The core isomorphism gives \eqref{qfig:eq-B-two}.
Finally the coefficients of the plane factor are $F_0=1$, $F_1=1$, $F_2=1+q^2$.
Using $K_0=1$ and \eqref{qfig:eq-length-one-K} in $B_G=F K_G$,
\[ K_2=B_2-F_1K_1-F_2K_0 =q^2+3q^4+q^6.
\]
\end{proof}

The distinction between \eqref{qfig:eq-B-two} and \eqref{qfig:eq-K-two} is essential.
The projective threefold $D_r$ is a family of local algebra structures, not a family of extra support points.
Its top class already forces a degree-six class in the ordinary cohomology.
The collision character is then obtained by the support decomposition, rather than by identifying all cohomology of $D_r$ or $P_2$ with an intersection image.

\subsection{Higher collisions: evaluation lattices and finite jets}
\label{qfij:section}

\begin{example}
    Unfortunately, two lower collisions does not determine the higher collision.
If $I,J\subset R$ are proper $\mathfrak m$-primary ideals, then
\begin{equation*}
 0\longrightarrow R/(I\cap J)\longrightarrow R/I\oplus R/J \longrightarrow R/(I+J)\longrightarrow0.
\end{equation*}
Since $I+J\subset\mathfrak m$, the length of the union is strictly smaller than the sum of the lengths.
Nor is $R/I\oplus R/J$ cyclic: its quotient by $\mathfrak m$ has dimension two.
A collision of prescribed total length must retain additional data.

For the following scalar constructions assume $r\geq3$.
Take the two axis clusters
\[ I_x=(y,x^r),\qquad I_y=(x,y^r).
\]
Their union $(xy,x^r,y^r)$ has length $2r-1$.
For a projective relative rate $[\alpha:\beta]$, consider
\begin{equation*}
 \begin{split}
 J_{\alpha,\beta}(\epsilon)= (&xy,\ \beta x^r+\alpha y^r-\alpha\beta\epsilon,\\ &x^{r+1}-\alpha\epsilon x,\ y^{r+1}-\beta\epsilon y).
 \end{split}
\end{equation*}
On the chart $\alpha=1$ the monic relations
\[ xy,\qquad y^r+\beta x^r-\beta\epsilon,\qquad x^{r+1}-\epsilon x \]
have standard monomials
\begin{equation*}
 1,x,\ldots,x^r,\quad y,\ldots,y^{r-1}.
\end{equation*}
Their overlapping reductions agree, so they give a free basis.
The other chart follows by interchanging $x,y$.
The parameter $\epsilon$ rescales between projective rate charts;
the central family itself is intrinsically parametrized by $\mathbb P^1$.

For $\alpha\beta\epsilon\ne0$, these are the two free orbits $(y,x^r-\alpha\epsilon)$ and $(x,y^r-\beta\epsilon)$.
Their central limit is
\begin{equation*}
 I_{[\alpha:\beta]}= (xy,\ \beta x^r+\alpha y^r,\ x^{r+1},y^{r+1}).
\end{equation*}
The last two generators are necessary at the projective endpoints.
All these algebras have regular length two.
They fit into
\[ 0\longrightarrow\CC\longrightarrow R/I_{[\alpha:\beta]} \longrightarrow R/(I_x\cap I_y)\longrightarrow0,
\]
where the kernel is an invariant socle line.
The projective rate specifies the surviving quotient of $\langle x^r,y^r\rangle$.
Thus fixing the two input clusters still leaves a projective line of possible outputs.\qed
\end{example}

Put $N=n|G|$, $\mcO=\CC[[s]]$, and $K=\CC((s))$.
A configuration of $n$ distinct free orbits over $K$, tending to the origin, has finite \'etale coordinate algebra $E$ of degree $N$.
Its collision algebra is obtained before specialization:
\[ L=\operatorname{im}(\mcO[x,y]\longrightarrow E).
\]
\begin{prop}[Flat collision algebra]
\label{qfij:flat}
The algebra $L$ is finite free of rank $N$.
It is the unique flat extension of the generic quotient, and $L/sL$ is a punctual cyclic quotient of charge $n\delta$.
Formation of $L$ commutes with a ramified splitting extension.
Every point of the principal punctual fibre is obtained in this way.
\end{prop}
\begin{proof}
The characteristic polynomials of $x$ and $y$ in $E$ have integral coefficients: they are symmetric functions of the coordinates of the orbit configuration.
Thus $x,y$ are integral, and the monomials $x^ay^b$, $0\le a,b<N$, generate a finite module containing $L$.
More directly their span is $L$, by the two characteristic equations.
It is torsion-free in $E$, hence free of rank $N$ over the DVR.
At the central point both characteristic polynomials are $T^N$;
the commuting coordinate operators are nilpotent.
Their cyclic quotient is therefore punctual.
Isotypic projectors split the free family, so each multiplicity is constant.

A finite flat quotient with generic ideal $I_K$ has ideal $I_K\cap\mcO[x,y]$, because its quotient is torsion-free.
This proves uniqueness.
Flat base change commutes with the image sequence.
If $s=z^\ell$ splits the points, then
\[ L\otimes_{\mcO}\CC[[z]] =\operatorname{im}(\CC[[z]][x,y]\ra\CC((z))^N) \subset\CC[[z]]^N,
\]
whose fibre modulo $z$ is canonically $L/sL$.
Finally choose a curve through a point of the closure of the free-orbit locus, with generic point in that locus; normalization and completion give the required arc.
\end{proof}

Consider the principal component of $\Hilb^n[\CC^2/G]$.
List the split points as $q_a(z)$ and form the finite matrix
\[ \mathsf E(z)_{a,f}=f(q_a(z)),\qquad f\in\{x^iy^j:i+j<N\}.
\]
Let $c$ be the minimum finite valuation of its $N$-column minors,
and let $C$ be a submatrix attaining that minimum.
\begin{lem}[Finite reconstruction]
\label{qfij:minors}
The columns of $C$ are an integral basis of the evaluation lattice.
In this basis its coordinate operators and unit are
\begin{equation}
\label{qfij:matrices}
B_x=C^{-1}\operatorname{diag}(x(q_a))C,\qquad B_y=C^{-1}\operatorname{diag}(y(q_a))C,\qquad v=C^{-1}\mathbf1.
\end{equation}
Their reductions modulo $z^k$ depend only on the point coordinates modulo $z^{c+k}$.
In particular $k=1$ recovers the limiting algebra.
\end{lem}
\begin{proof}
In an $N$-dimensional algebra generated by two elements, the degree filtration begins in dimension one.
Once it stops growing it is stable under both generators, hence is the whole algebra.
Therefore degrees below $N$ span the special fibre, and Nakayama's lemma shows that the displayed columns generate the lattice.
Cramer's rule and minimal valuation show that every column has integral coordinates in the selected basis.
The multiplication matrices preserve the lattice and are integral as well.

If the coordinates change by $z^{c+k}$, both the basis and border evaluation matrices change by that order.
In $B'-B=\frac{1}{C'}(E'_b-E_b-(C'-C)B)$ the parentheses have order at least $c+k$, and the inverse has entries of order at least $-c$.
The difference thus has order at least $k$.
This also applies to the unit.
The monodromy permutes all evaluation rows; these permutations cancel in \eqref{qfij:matrices}, proving descent.
\end{proof}
On a fixed pivot and contact stratum these are regular formulas.
Reparametrizing an arc can increase $c$ arbitrarily, so this is a finite-jet description at bounded contact,
not a uniform jet bound depending only on $n$.
Conditions on the Hilbert function, socle,
or ranks of products of $B_x,B_y$ are constructible conditions on these finite jets.
Thus the retained multiplication explains which additional directions, beyond the limiting support, survive a collision.

\subsection{Motivic integration and recovery of the Betti series}
\label{subsec:collision-integration}
Set $W=(\CC^2)^n$, $H=G^n\rtimes S_n$, $d=2n$, and $X_n=W/H$.
We compare the two proper birational maps
\[ \pi:\Hilb^n[\CC^2/G]\longrightarrow X_n,
 \qquad \rho:\mathcal S_n=[W/H]\longrightarrow X_n.
\]
Both are isomorphisms over distinct free-orbit configurations.
A twisted arc in the sector of $h\in H$ of order $\ell$ is represented by $\mathbf p(\zeta z)=h\mathbf p(z)$,
with $s=z^\ell$, modulo $C_H(h)$.
If the eigenvalues of $h$ are $\zeta^{b_i}$, $0\le b_i<\ell$, its coordinates are $p_i(z)=z^{b_i}a_i(s)$.
Use the shift convention
\[ \operatorname{sht}(h)=\sum_{b_i>0}(1-b_i/\ell) =\operatorname{age}(h^{-1}).
\]
In Yasuda's completed coefficient theory, the measure of an $m$-jet cylinder on a smooth $d$-dimensional stack is $[C]\LL^{-md}$.
Thus arcs on a smooth scheme centered on $T$ have measure $[T]$.
A quotient jet stratum has class $[C/C_H(h)]$; its rational Hodge realization is that of the coarse quotient,
not $[C]/|C_H(h)|$.

Let $\mathcal A_n$ be the twisted arcs in $\mathcal S_n$ whose coarse arcs are centered at $n[0]$ and generically over the free-configuration locus.
Properness gives a unique lift $\widetilde\gamma$ to the Hilbert space, and the evaluation lattice gives its center $\eta(\alpha)$.
Define
\[ a=\operatorname{ord}_s\operatorname{Jac}_\pi(\widetilde\gamma),
 \qquad b=\operatorname{ord}_z\operatorname{Jac}_\rho(\mathbf p),
 \qquad \nu=\operatorname{sht}(h)+a-b/\ell,
 \qquad \operatorname{Jac}_f=\operatorname{Fitt}_0\Omega_f.
\]
\begin{theorem}[Refined collision integral]
\label{qfij:main}
For every constructible subset $T\subset P_n$,
\begin{equation}
\label{qfij:main-equation}
[T]=\int_{\eta^{-1}(T)}\LL^{\nu(\alpha)}\,d\mu_{\mathcal S_n}(\alpha).
\end{equation}
The identity holds in the completed motivic coefficient theory and after compactly supported Hodge realization.
The integrand and the limiting-algebra condition are finite-jet conditions on every stratum of bounded pivot contact and Jacobian orders.
\end{theorem}
\begin{proof}
The change-of-variables theorem for a proper tame birational map $f:\mathcal Y\ra\mathcal Z$, with smooth source and reduced variety target,
is \cite[Theorem 3.41]{Yasuda_2004c}
\[ \int_A\Phi\,d\mu_{\mathcal Z} =\int_{f_\infty^{-1}(A)}(\Phi\circ f_\infty) \LL^{-\operatorname{ord}\operatorname{Jac}_f+ \operatorname{sht}_{\mathcal Y}}d\mu_{\mathcal Y}.
\]
Take $A$ to be the coarse arcs whose Hilbert lift is centered on $T$ and generically free, and take $\Phi=\LL^{\operatorname{ord}\operatorname{Jac}_\pi}$.
For $f=\pi$, its Jacobian factor cancels $\Phi$, giving $[T]$;
arcs lying entirely in the proper closed boundary have measure zero.
For $f=\rho$, the same integral becomes the right side of \eqref{qfij:main-equation}.
The transformation theorem also gives convergence.

For explicit finite jets, use the basis chart $[X,Y]=0$, $b_i(X,Y)v=e_i$, with equivariance imposed, and choose local smooth coordinates on that chart.
Choose a finite generating collection of invariant power sums for $\CC[W]^H$ as support coordinates; their pullbacks are $|G|^{-1}\operatorname{Tr} f(X,Y)$.
The maximal minors of their differentials compute $\operatorname{Jac}_\pi$.
On $W$, the maximal minors of the differentials of the same invariant power sums compute $\operatorname{Jac}_\rho$.
Differentiation is taken in the ambient smooth chart before evaluation on the arc.
Orders $a,b$ are finite determinantal conditions.
Together with Lemma \ref{qfij:minors}, a truncation with $m\ell\ge\max\{c+\ell(a+1),b+1\}$ suffices.
Finally each quotient Jacobian minor transforms by $\det\frac{1}{h}$.
Its valuation is congruent to $-\sum_i b_i$ modulo $\ell$; hence $\nu$ is integral.
\end{proof}

In particular, we have the following evaluations:

\begin{corollary}
\label{qfij:betti}
Set $\mathcal Z_n=\int_{\mathcal A_n}\LL^\nu d\mu_{\mathcal S_n}$ and $\mathcal Z_0=1$.
Then
\begin{equation*}
 B_G(q,t)=\sum_{n\ge0}E_c(\mathcal Z_n;-q,-q)t^n,
 \qquad K_G(q,t)=F(q,t)^{-1} \sum_{n\ge0}E_c(\mathcal Z_n;-q,-q)t^n.
\end{equation*}
\end{corollary}
\begin{proof}
First identify the integral with $[P_n]$.
Proper-core purity makes $E_c(P_n;-q,-q)=P(P_n;q)$, and the core isomorphism gives $B_G$.
The support theorem then gives $K_G$.
Thus these additive data determine the ranks of the compact intersection maps.
Their actual pairings still require the normal embedding or supported Ext calculation of Section \ref{qfig:sec-intersection}.
\end{proof}

\section{Hall homology action on Hilbert cohomology}
\label{sec:hall-action}
We use the convention that $\QQ(1)$ has weight $-2$, and
for the mixed Hodge structures on Borel-Moore homology and their Tate normalizations,
see \cite[Sections 2.1.2 and 2.3.2]{DavisonPurity2CY}.
\subsection{Hall algebras and their Borel-Moore homology}
\label{subsec:hall-preliminaries}

Put $A=R\rtimes G$, the skew group algebra with multiplication $(rg)(sh)=r\,g(s)gh$.
A left $A$-module is equivalently an $R$-module with a compatible $G$-action.
Let $\mathscr C_G^0$ be the category of finite-dimensional left $A$-modules supported at the origin.
Its exact sequences are the ordinary short exact sequences of modules.
It is closed under subobjects, quotients, and extensions: if $\mathfrak m^aQ_1=0$ and $\mathfrak m^bQ_2=0$,
an extension of $Q_2$ by $Q_1$ is killed by $\mathfrak m^{a+b}$, where $\mathfrak m=(x,y)$.
The charge of an object is its dimension vector $\bfw$,
which is additive in exact sequences.

Write $V_{\bfw}=\bigoplus_i\rho_i^{w_i}$ and $GL_\bfw=\prod_iGL(w_i)$.
The relevant moduli stacks are
\[
\begin{aligned}
 \operatorname{Rep}_{\bfw}(A) &=\{(X,Y)\in\Hom_G(U\otimes V_{\bfw},V_{\bfw}):XY=YX\},\\ \mathfrak R_{\bfw} &=[\operatorname{Rep}_{\bfw}(A)/GL_\bfw].
 \end{aligned}
\]
In this notation $(X,Y)$ denotes a single equivariant multiplication map, written in the basis $x,y$ of $U$.
The quotient stack classifies finite locally free families of $A$-modules with the specified isotypic ranks: frames of their multiplicity bundles give its atlas,
and descent recovers the families.

For $N=|\bfw|>0$, let $Z^0_{\bfw}$ be the reduced closed locus on this atlas defined by
$ \mathrm{Tr}(X^aY^b)=0 \text{ for } 1\leq a+b\leq N,$ and further define the stack
 \[ \mathfrak T_{\bfw}=[Z^0_{\bfw}/GL_\bfw].\]
Its geometric modules are precisely those supported at the origin.
Indeed, Newton identities applied to $sX+tY$ give $(sX+tY)^N=0$; comparison of coefficients gives $(x,y)^NQ=0$.
The converse holds over fields.
Set $\mathfrak T_{\mathbf0}=\operatorname{Spec}\CC$.
These are Artin stacks, and every nonzero module has scalar automorphisms.
Specifying an invariant cyclic generator removes all automorphisms and gives the punctual Quot fibre $P_{\bfw}$.
The reduction convention does not change rational Borel-Moore homology.
It avoids identifying geometric punctuality with an assertion about all infinitesimal families.

\begin{defn}[Stack homology and Hall homology]
\label{def:hall-homology}
For a quotient stack $\mathfrak X=[Z/H]$, use the convention
\begin{equation*}
 H_j^{\mathrm{BM}}(\mathfrak X) =\mathbb H^{-j}(\mathfrak X,\omega_{\mathfrak X}).
\end{equation*}
Equivalently, choose an $l$-dimensional representation $W$ of $H$ with a free open subset $W^\circ$.
In the stable range,
as the codimension of its complement increases, this is
\[ H_{j+2l}^{\mathrm{BM}}((Z\times W^\circ)/H)(-l).
\]
The twist is part of the mixed-Hodge convention; it does not change the underlying rational vector space.
The \emph{punctual Hall homology} is the charge-graded vector space
\[ \mathcal H_G^0=\bigoplus_{\bfw,j} H_j^{\mathrm{BM}}(\mathfrak T_{\bfw}).
\]
Its Hall multiplication is defined below.
The superscript $0$ records support at the origin, not homological degree.
\end{defn}

For a smooth stack of complex dimension $d$,
$\omega_{\mathfrak X}=\QQ_{\mathfrak X}[2d](d)$, so that $H_j^{\mathrm{BM}}(\mathfrak X)=H^{2d-j}(\mathfrak X)(d)$.
Dimensions and homological degrees can be negative.
For example,
$H^*(B\mathbb G_m)=\QQ[c]$, $|c|=2$, whereas
\[ H_j^{\mathrm{BM}}(B\mathbb G_m)=
\begin{cases}
\QQ(-a-1),&j=-2-2a,\ a\geq0,\\ 0,&\text{otherwise}.
 \end{cases}
\]
The stacks $\mathfrak T_{\bfw}$ are generally singular, so their dualizing complexes cannot be replaced by shifted constant sheaves.
Every fixed-degree group is finite-dimensional and carries its natural mixed Hodge structure.

We use proper pushforward, smooth pullback, and refined Gysin maps with their usual base-change and composition identities.
For rational Chow groups of Deligne--Mumford stacks this is Vistoli's intersection-theoretic formalism \cite{Vistoli1989}.
The Hall stacks here have positive-dimensional stabilizers;
for their Borel-Moore homology we use the corresponding operations and virtual pullbacks of \cite[Section 3]{KapranovVasserot}.
In local zero-locus presentations the virtual operation is a smooth pullback followed by a refined Gysin map for a vector-bundle section.
Thus the correspondence arguments below reduce to compatibility of the extension complexes and these functorial identities.

To define multiplication, let $\mathfrak E^0_{\bfw_1,\bfw_2}$ be the reduced classical stack of short exact sequences $0\ra Q_1\ra Q\ra Q_2\ra0$ in $\mathscr C_G^0$,
with $Q_i$ of charge $\bfw_i$.
There is a correspondence
\begin{equation*}
 \mathfrak T_{\bfw_1}\times\mathfrak T_{\bfw_2} \xleftarrow{\ p\ }\mathfrak E^0_{\bfw_1,\bfw_2} \xrightarrow{\ m\ }\mathfrak T_{\bfw_1+\bfw_2},
 \qquad p=(Q_1,Q_2),\quad m=Q.
\end{equation*}
Thus the first factor is the submodule.
The map $m$ is representable and proper: over $Q$, it parametrizes invariant submodules of a specified charge,
a closed locus in a product of Grassmannians.
The map $p$ need not be smooth.
Its pullback requires the following extension complex, rather than an ordinary smooth pullback.

For fixed endpoints, the equivariant Koszul resolution gives
\[ C^p=\Hom_G(\Lambda^pU\otimes Q_2,Q_1),\qquad 0\leq p\leq2,
 \qquad H^p(C^\bullet)=\Ext_A^p(Q_2,Q_1).
\]
Writing $X_i,Y_i$ for multiplication on $Q_i$, the differentials are explicitly
\[
\begin{aligned}
 D_0(f)&=(X_1f-fX_2,\ Y_1f-fY_2),\\ D_1(\xi_x,\xi_y) &=X_1\xi_y+\xi_xY_2-Y_1\xi_x-\xi_yX_2.
 \end{aligned}
\]
Commutativity on the endpoints implies $D_1D_0=0$.
A $G$-equivariant splitting of the underlying vector spaces exists by averaging.
In such a splitting, the off-diagonal multiplication is $\xi\in C^1$, commutativity is $D_1\xi=0$, and a change of splitting adds $D_0f$.
Thus the classical fibre is $[\ker D_1/C^0]$.
The derived zero locus of $D_1$, divided by $C^0$, retains the obstruction space $\Ext_A^2(Q_2,Q_1)$.
It defines the virtual pullback $p^!$ of relative dimension
\[ d_p=\dim C^1-\dim C^0-\dim C^2 =-E(\bfw_2,\bfw_1).
\]
This is the surface extension construction of \cite[Sections 2-4]{KapranovVasserot}.
Taking $G$-invariants is exact, so the same complexes and their flag compatibilities apply here.
The punctual virtual operation is inherited by base change from the ambient extension construction.
Passing to reduced classical stacks makes the middle-term map factor through the reduced punctual target and does not change Borel-Moore homology.
No additional obstruction theory is assigned to the nilpotence equations.

\begin{defn}[Punctual Hall product]
\label{def:hall-product}
For $\alpha_i\in H_{j_i}^{\mathrm{BM}}(\mathfrak T_{\bfw_i})$,
set
\begin{equation*}
 \alpha_1\star\alpha_2=m_*p^!(\alpha_1\boxtimes\alpha_2).
\end{equation*}
It has homological degree $j_1+j_2-2E(\bfw_2,\bfw_1)$.
More precisely, multiplication is a mixed-Hodge morphism
\[ H_{j_1}^{\mathrm{BM}}(\mathfrak T_{\bfw_1})\otimes H_{j_2}^{\mathrm{BM}}(\mathfrak T_{\bfw_2}) \longrightarrow H_{j_1+j_2-2E_{21}}^{\mathrm{BM}} (\mathfrak T_{\bfw_1+\bfw_2})(E_{21}),
 \qquad E_{21}=E(\bfw_2,\bfw_1).
\]
We suppress this target Tate twist when speaking only of rational vector spaces.
The resulting algebra is denoted $\mathcal H_G^0$.
\end{defn}

Associativity follows by applying the same construction to a three-step filtration: both parenthesizations forget the intermediate submodules of the same flag.
The extension exact triangles identify the composite virtual pullbacks, and proper base change identifies the composite pushforwards.
This is the flag compatibility of \cite[Section 4.5]{KapranovVasserot}; see also \cite[Part II]{DiaconescuPortaSala}.
The class of the zero module is the unit.
Replacing $\mathfrak T_{\bfw}$ by $\mathfrak R_{\bfw}$ defines the Hall algebra $\mathcal H_G$ with arbitrary finite support.
Proper pushforward from the closed punctual loci gives an algebra homomorphism $\mathcal H_G^0\ra\mathcal H_G$,
not an asserted injection.

The algebra used for the regular Betti series is the subalgebra
\begin{equation*}
 \mathcal A_G^0=\bigoplus_{n,j}H_j^{\mathrm{BM}}(\mathfrak T_{n\delta}),
 \qquad \mathcal A_{n,j}=H_j^{\mathrm{BM}}(\mathfrak T_{n\delta}).
\end{equation*}
We also write $\mathcal A=\mathcal A_G^0$.
The identity $E(n\delta,m\delta)=0$ makes both $n$ and $j$ additive under multiplication, with no Tate correction.
In the centered grading $e=-j$, the same statement holds for $(n,e)$.
The charge-zero part is $\QQ$; projection onto it is the augmentation $\epsilon:\mathcal A\ra\QQ$, and $\mathcal A_+=\ker\epsilon=\bigoplus_{n>0,j}\mathcal A_{n,j}$.
The algebra is a direct sum.
Its charge completion is needed only when infinite formal sums are used; multiplication in a fixed charge involves finitely many charge splittings.

Here ``Hall homology'' means the Borel-Moore groups with this extension product.
It does not mean the homology of the algebra's bar complex.
The latter is a free resolution of a specified left module $L$; applying $\QQ\otimes_{\mathcal A}-$,
with $\QQ$ defined by $\epsilon$, computes $\operatorname{Tor}^{\mathcal A}_*(\QQ,L)$.
Likewise, the augmentation ideal $\mathcal A_+$ differs from the vacuum annihilator: the vacuum map takes values in the collision module,
and positive-charge operators can act nontrivially on the vacuum.
The annihilator is a left ideal, so its quotient is a cyclic module;
it is not in general a quotient algebra.
These distinctions will be used in Section \ref{qfih:section}.

\subsection{Punctual module stacks and their action}
\label{subsec:punctual-action}

The compact intersection images are compatible across charges with modification of ideals.
We now describe the action of the Hall algebra just defined on the full cohomology before restricting it to the collision spaces.

\begin{prop}[The modification action]
\label{prop:punctual-action}
The algebra $\mathcal H_G^0$ acts on $\bigoplus_{\mathbf v}H^*(\Quot^{\mathbf v}[\CC^2/G])$.
Its regular-charge subalgebra $\mathcal A_G^0$ acts on
\[ \mathbb V_G=\bigoplus_{n\geq0}H^*(\Hilb^n[\CC^2/G]).
\]
For $\alpha\in H_j^{\mathrm{BM}}(\mathfrak T_{k\delta})$, the operation has degree
\begin{equation}
\label{eq:punctual-action-degree}
\rho(\alpha):H^d(\Hilb^n[\CC^2/G]) \longrightarrow H^{d+2k-j}(\Hilb^{n+k}[\CC^2/G]).
\end{equation}
It is compatible with compact supports and preserves $\mathcal I_G=\bigoplus_{n\geq0}I_n(G)$.
Consequently $B_G$ is the character of this full representation, and $K_G$ is the character of its specified invariant subspace $\mathcal I_G$.
\end{prop}

\begin{proof}
Let $\mathcal Z_{\bfw,\mathbf v}$ parametrize nested invariant ideals $J'\subset J\subset R$, where $R/J$ has charge $\mathbf v$ and $S=J/J'$ is punctual of charge $\bfw$.
There is a correspondence
\begin{equation*}
 \mathfrak T_{\bfw}\times\Quot^{\mathbf v}[\CC^2/G] \xleftarrow{\ p\ }\mathcal Z_{\bfw,\mathbf v} \xrightarrow{\ t\ }\Quot^{\mathbf v+\bfw}[\CC^2/G],
 \qquad p=(S,J),\quad t=J'.
\end{equation*}
The relative construction over $(S,J)$ is the surjective locus in the derived space of maps $J\ra S$.
Its deformation complex is $\operatorname{RHom}_A(J,S)$.
An equivariant ideal has projective dimension at most one: this follows from the two-dimensional Koszul resolution and exactness of taking $G$-invariants.
The complex therefore has amplitude $[0,1]$, so it defines the virtual pullback $p^!$.
Its relative rank is
\begin{equation*}
 r(\bfw,\mathbf v) =\chi_A(J,S)=w_0-E(\mathbf v,\bfw),
\end{equation*}
by $0\ra J\ra R\ra R/J\ra0$ and $\Hom_A(R,S)=S^G$, with $R$ projective as an $A$-module.

The map $t$ is projective: for fixed $J'$, the possible $J/J'$ form a closed invariant Grassmannian inside $R/J'$,
with its punctual condition closed.
Thus, on Borel-Moore homology,
the operation is explicitly
\begin{equation}
\label{eq:punctual-convolution}
\rho(\alpha)(b)=t_*p^!(\alpha\boxtimes b).
\end{equation}
This specifies the virtual operation; using only the fundamental class of the classical nested-ideal scheme would omit its obstruction term.

For flags $J''\subset J'\subset J$,
the exact sequence
\[
0\longrightarrow J'/J''\longrightarrow J/J''\longrightarrow J/J'\longrightarrow0
\]
identifies the two composite extension complexes through its $\operatorname{RHom}_A$ triangle.
Composition of the associated virtual pullbacks and their compatibility with proper pushforward therefore give
\[
\rho(\alpha\star\beta)=\rho(\alpha)\rho(\beta).
\]
This is the modification version of the flag argument in \cite[Section 4.5]{KapranovVasserot};
see also \cite[Theorem III.2.1 and Section III.2.1.1]{DiaconescuPortaSala}.
Exactness of $G$-invariants gives the same identities equivariantly.
The zero modification acts as the identity.

For balanced charges $E(n\delta,k\delta)=0$ and $r(k\delta,n\delta)=k$.
Hence \eqref{eq:punctual-convolution} sends homological degree $b$ to $b+j+2k$.
Poincar\'e duality on the smooth Hilbert schemes of complex dimensions $2n$ and $2(n+k)$ gives \eqref{eq:punctual-action-degree}.
For general charges the same argument gives cohomological shift $2D(\mathbf v+\bfw)-2D(\mathbf v)-j-2r(\bfw,\mathbf v)$.

The projection $s(J',J)=J$ is also projective:
a punctual quotient of length $N=|\bfw|$ factors through $J/\mathfrak m^N J$,
so the reduced parameter space is a closed relative Grassmannian locus.
With the correspondence class
\[
\gamma_\alpha=p^!\bigl(\alpha\boxtimes[\Quot^{\mathbf v}[\CC^2/G]]\bigr),
\]
the operator is $b\mapsto t_*(\gamma_\alpha\cap s^*b)$ after duality.
The projection formula and properness of both $s$ and $t$ give the same operation on compactly supported cohomology,
compatibly with $H_c^*\to H^*$.
Its image $\mathcal I_G$ is therefore preserved.
\end{proof}

\begin{remark}[The representation stack and punctual support]
The same construction with $\mathfrak R_{\bfw}$ in place of $\mathfrak T_{\bfw}$ allows modifying modules with arbitrary finite support.
The projection to $J'$ is still projective, so it acts on the Borel-Moore homology and hence the ordinary cohomology of the Quot schemes.
The closed punctual inclusion induces a Hall homomorphism by proper pushforward, and its action is the restriction just constructed.
Punctuality is what also makes the projection to $J$ proper and gives the stated compact-support compatibility.
No cyclicity or determinant-one hypothesis on $G$ was used.
\end{remark}

At the vacuum $1\in H^0(\Hilb^0[\CC^2/G])$, the relative complex is the vector bundle $S^G$.
Its surjective locus consists of invariant cyclic vectors.
If $\mathfrak C_{n\delta}\subset\mathfrak T_{n\delta}$ denotes the open substack admitting such a vector,
and $f:P_n\ra\mathfrak C_{n\delta}$ forgets it, then $f$ is smooth of relative dimension $n$.
The vacuum operation is the composite
\begin{equation}
\label{eq:punctual-vacuum}
\begin{split}
 H_j^{\mathrm{BM}}(\mathfrak T_{n\delta}) &\longrightarrow H_j^{\mathrm{BM}}(\mathfrak C_{n\delta}) \xrightarrow{f^*}H_{j+2n}(P_n)\\ &\xrightarrow{i_*}H_{j+2n}^{\mathrm{BM}}(\Hilb^n[\CC^2/G]) \cong H^{2n-j}(\Hilb^n[\CC^2/G]).
 \end{split}
\end{equation}
Its last map is precisely the compact intersection map after duality.
Thus every punctual vacuum class lies in $I_n^{2n-j}(G)$.
A Hall class whose framed pullback lies in the intersection radical annihilates the vacuum.
This relates the modification construction to the geometric definition of $K_G$, without identifying the support factorization with a tensor factorization of representations.
The algebra action alone does not evaluate its character or determine the intersection matrices.

\subsection{The Hall action on a general smooth surface stack}
\label{subsec:global-hall-action}

Let $\mcX$ be a smooth orbisurface over $\CC$.
Work with finite-type fixed-class Quot and coherent-sheaf moduli, and with projective modification maps.
For example, these conditions hold in the quasi-projective coarse-space setting with a generating sheaf and fixed modified lengths.
Let $\mathfrak C_\beta(\mcX)$ be the Artin stack of zero-dimensional coherent sheaves of class $\beta$ as introduced in section \ref{subsec:alpha-type}, with proper finite support.
When the type data do not determine an Euler pairing, refine the connected components until the Euler ranks below are constant.
There is no need to choose a global quotient presentation of $\mcX$.

Write $\mathfrak E_{\beta_1,\beta_2}$ for the stack of exact sequences $0\ra S_1\ra S\ra S_2\ra0$.
Its endpoint and middle-term maps are
\[ \mathfrak C_{\beta_1}(\mcX)\times\mathfrak C_{\beta_2}(\mcX) \xleftarrow{\ p\ }\mathfrak E_{\beta_1,\beta_2}
\xrightarrow{m}\mathfrak C_{\beta_1+\beta_2}(\mcX).\]
The extension complex $R\Hom_{\mcX}(S_2,S_1)$ has amplitude $[0,2]$.
The associated surface virtual pullback has relative rank $-\chi_{\mcX}(S_2,S_1)$; its degree is constant on each component under consideration.
The map $m$ is projective, being the fixed-class relative Quot space of the universal finite sheaf over the middle-term stack.
Consequently
\[ a_1\star a_2=m_*p^!(a_1\boxtimes a_2) \]
defines the Hall multiplication on $\bigoplus_\beta H_*^{\mathrm{BM}}(\mathfrak C_\beta(\mcX))$.
Its homological shift is $-2\chi_{\mcX}(S_2,S_1)$, with target Tate twist $\chi_{\mcX}(S_2,S_1)$.
The zero sheaf supplies the unit.
Associativity follows from the extension triangles for three-step flags.
These are the surface Hall constructions of \cite{KapranovVasserot,DiaconescuPortaSala}; the finite stabilizer invariants in each quotient chart are exact.
Their perfect complexes and flag identifications descend, giving the stated stack version.

Let $H_\alpha$ and $H_{\alpha+\beta}$ be smooth connected components of $\Quot^\alpha\mcX$ and $\Quot^{\alpha+\beta}\mcX$,
of complex dimensions $D_\alpha$ and $D_{\alpha+\beta}$.
Nested ideals define
\[ \mathfrak C_\beta(\mcX)\times H_\alpha \xleftarrow{\ p_{\mathrm{mod}}\ }\mathcal Z_{\beta,\alpha} \xrightarrow{\ t\ }H_{\alpha+\beta},
 \qquad J'\subset J,\quad S=J/J'.
\]
The relative construction is the surjective locus of the derived space of maps $J\ra S$.
Its deformation complex $R\Hom_{\mcX}(J,S)$ has amplitude $[0,1]$: ideals on a smooth surface have local projective dimension at most one,
and finite support eliminates higher sheaf cohomology.
Thus
\[ r=\chi_{\mcX}(J,S) =\chi_{\mcX}(\mcO_{\mcX},S) -\chi_{\mcX}(\mcO_{\mcX}/J,S) \]
is the virtual relative rank, on each component with fixed rank.
For $a\in H_j^{\mathrm{BM}}(\mathfrak C_\beta(\mcX))$, the operator
\begin{equation*}
 \rho(a)(b)=t_*p_{\mathrm{mod}}^!(a\boxtimes b)
\end{equation*}
has Borel-Moore degree shift $j+2r$.
The map $t$ is projective:
relatively over the target, it is the fixed-class closed locus in the relative Quot space of the universal finite sheaf $\mcO_{\mcX}/J'$.
Poincar\'e duality on the two smooth components therefore yields the precise ordinary-cohomology operation
\begin{equation}
\label{eq:global-hall-degree}
H^d(H_\alpha)\longrightarrow H^{d+2D_{\alpha+\beta}-2D_\alpha-j-2r}(H_{\alpha+\beta}).
\end{equation}
Flags $J''\subset J'\subset J$ and their $R\Hom$ triangles identify composition with Hall multiplication.
This proves the action on the total cohomology, with componentwise dimension and Euler shifts.
For balanced charges $n\delta,k\delta$ in $[\CC^2/G]$, the ranks specialize to $D_{n\delta}=2n$ and $r=k$,
recovering the shift $d\mapsto d+2k-j$.

If modifying sheaves are required to lie in a fixed proper closed subset $C\subset\mcX$, take the reduced closed-support loci,
as in the punctual convention above.
The projection $s(J',J)=J$ is proper as well.
On these reduced loci, for bounded length the quotient is supported on a bounded infinitesimal neighborhood of $C$,
and its parameter space is a projective relative Quot space.
Hence the correspondence acts on compactly supported cohomology compatibly with $H_c^*\ra H^*$ and preserves its image.
This applies in particular to modifications at one stacky point.
For unrestricted support on a nonproper surface, $s$ need not be proper; the ordinary-cohomology action remains defined,
but this compact-support conclusion is not asserted.
None of these constructions requires a cyclic stabilizer or a determinant-one tangent representation.

\subsection{Hyperbolic coefficients, the vacuum, and bounded resolutions}
\label{qfih:section}
The local action has two additional structures useful for computing its character, and We record their precise scope.
Write $g=|G|$, $\mathfrak T_n=\mathfrak T_{n\delta}=[Z_n/H_n]$,
$h_n=\dim H_n=gn^2$, $z_n=\dim Z_n$, and $\tau_n=z_n-h_n$.
Scalar multiplication contracts the punctual matrix variety $Z_n$ to zero.
It gives $c_n:\mathfrak T_n\ra BH_n$ and the zero section $i_n:BH_n\ra\mathfrak T_n$.

\begin{prop}[Dualizing coefficient localization]
\label{qfih:prop:coefficient}
With $\mathscr D_n=i_n^*\omega_{\mathfrak T_n}$, one has
\begin{equation*}
 Rc_{n*}\omega_{\mathfrak T_n}\simeq\mathscr D_n,
 \qquad H_j^{\mathrm{BM}}(\mathfrak T_n) =\mathbb H^{-j}(BH_n,\mathscr D_n).
\end{equation*}
Moreover $\mathscr D_n$ has amplitude $[-2\tau_n,2h_n]$, and
\begin{equation*}
 H_j^{\mathrm{BM}}(\mathfrak T_n)=0\quad(j>2\tau_n),
 \qquad \tau_n\le gn^2.
\end{equation*}
\end{prop}
\begin{proof}
On each free equivariant approximation the contraction lemma gives $Rc_{n!}\QQ\simeq i_n^!\QQ$; Verdier duality gives the displayed identity \cite{BradenHyperbolic,RicharzHyperbolic}.
Smooth descent to the atlas identifies the dualizing complex with $\omega_{Z_n}[-2h_n](-h_n)$.
Its stalk at zero has degrees between $-2z_n+2h_n$ and $2h_n$, by the local cohomological dimension bound.
Cohomology of $BH_n$ has nonnegative degrees.
Finally $Z_n$ is closed in $\Hom_G(U\otimes n\CC[G],n\CC[G])$, of dimension $2gn^2$, proving the upper bound.
\end{proof}

This is the opposite scalar hyperbolic restriction with dualizing coefficients.
The other direction gives $Rc_{n!}\omega=\omega_{BH_n}$ and loses the desired groups.
In particular $\mathscr D_n$ cannot be replaced by a constant sheaf: its attachment maps retain the singularities of the punctual cone.
The assertion concerns the specified equivariant approximations, not an identification of the intrinsic fixed stack (which also remembers module gradings) with $BH_n$.
No equivariant parameter is inverted.

Set $\mathcal A_{n,j}=H_j^{\mathrm{BM}}(\mathfrak T_n)$ and $\mathbb I_{n,j}=I_n^{2n-j}(n)$.
This Tate twist makes the latter pure of weight $-j$ and changes no Betti numbers.
\begin{prop}[Minimum weights and the vacuum]
\label{qfih:lem:weights}
The spaces $\mathcal A^{\min}_{n,j}=W_{-j}\mathcal A_{n,j}$ form a subalgebra, and the vacuum action is a surjection
\[
\nu:\mathcal A^{\min}\twoheadrightarrow\mathbb I,
\qquad \nu(a)=\rho(a)1,
\qquad \mathbb I\cong\mathcal A^{\min}/\operatorname{Ann}_{\mathcal A^{\min}}(1)
\]
of bigraded left modules.
Here $1\in\mathbb I_{0,0}=\QQ$ is the vacuum,
and its annihilator is a left ideal.
If $p_n=\dim P_n$ and $\kappa_n=2(p_n-n)$, then
\begin{equation*}
 \mathbb I_{n,j}=0\quad\hbox{unless}\quad |j|\le\kappa_n.
\end{equation*}
In particular $|j|\le2n$ is always necessary for a nonzero image.
\end{prop}
\begin{proof}
Borel-Moore homology in degree $j$ has weights at least $-j$;
the same bound follows on quotient-stack approximations after their dimension shift and Tate twist.
The Hall product preserves these bounds additively.

For $n>0$, the invariant-cyclic open substack is the neutral classifying stack $B_{P_n}((\mcQ^G)^\times)$.
A cyclic module determines its annihilator ideal, and changing its cyclic vector is multiplication by a unit.
Augmentation and the finite logarithm identify the unit group with $\mathbb G_m\times\mathbb V(\mathcal E)$,
of dimension $n$.
Consequently
\[ H_j^{\mathrm{BM}}(\mathfrak C_{n\delta}) =\bigoplus_{a\ge0}H_{j+2n+2a}(P_n)(-n-a).
\]
Every summand is pure of weight $-j$.
In the open-closed homology sequence, its boundary would have weight at least $1-j$ and hence vanishes.
Restriction from $\mathfrak T_n$ is onto; the framing pullback selects $a=0$ and is onto.
Its final pushforward has image $I_n$, as in \eqref{eq:punctual-vacuum}.
Strictness of weights proves surjectivity already on $W_{-j}$.
Finally the two arguments of the compact intersection pairing have homological degrees $2n+j$ and $2n-j$,
both at most $2p_n$.
This proves the window.
Since $p_n\le2n$, it also proves the unconditional bound.
\end{proof}

\begin{corollary}[The collision character]
\label{cor:collision-hall-character}
The collision series is the graded character of the cyclic minimum-weight Hall module generated by the vacuum:
\begin{equation}
\label{eq:collision-hall-character}
K_G(q,t)=\sum_{n\geq0}\sum_{j\in\ZZ}\dim\left(\mathcal A^{\min}/\operatorname{Ann}_{\mathcal A^{\min}}(1)\right)_{n,j}q^{2n-j}t^n.
\end{equation}
The character records ordinary graded dimensions, without alternating signs.
\end{corollary}
\begin{proof}
Proposition \ref{qfih:lem:weights} identifies the quotient in bidegree $(n,j)$ with $I_n^{2n-j}(G)(n)$.
The Tate twist preserves its dimension,
so the substitution $d=2n-j$ gives the defining series $K_G(q,t)=\sum_{n,d}\dim I_n^d(G)q^dt^n$.
\end{proof}

In the centered grading $e=-j$, equation \eqref{eq:collision-hall-character} is obtained from the character of $\mathbb I$ by replacing $t$ with $q^2t$.
This is a quotient of left modules; the vacuum annihilator need not be a two-sided ideal.
Evaluating this character still requires the annihilator or the compact intersection maps;
the character description alone does not give a product formula.

For each charge the Chern ring
\[ \mathscr R_n=H^*(BH_n)= \QQ[c_{i,a}:1\le a\le nd_i],\qquad |c_{i,a}|=2a,
\]
acts on the collision space through the universal multiplicity bundles.
In centered degree $e=-j$, that finite-dimensional module has a finite graded free resolution over this polynomial ring.
If $s_n=n\sum_i d_i$ and $d'_1\le\cdots\le d'_{s_n}$ are the variable degrees, the generators in its $l$-th minimal free term have degrees in
\begin{equation*}
 \left[-\kappa_n+\sum_{a=1}^l d'_a,
       \ \kappa_n+\sum_{a=s_n-l+1}^{s_n}d'_a\right],
 \qquad 0\le l\le s_n.
\end{equation*}
Indeed these generators are the corresponding Tor group, computed by the Koszul complex on the Chern variables; its exterior degrees give exactly these bounds.
This resolves a fixed-charge module over its Chern ring, not the vacuum annihilator over the Hall algebra.

Across charges use the normalized bar resolution $B_l=\mathcal A\otimes\mathcal A_+^{\otimes l}\otimes\mathbb I$,
whose faces multiply adjacent factors or act on the last factor.
The usual insertion-of-unit homotopy proves exactness of its augmentation to $\mathbb I$.
At total charge $N$, total internal degree $J$, and centered degree $e=-J$, put
\[ C_{N,0}=N^2,\qquad C_{N,l}=(N-l+1)^2+l-1\quad(1\le l\le N).
\]
Then
\begin{equation*}
 B_l=0\ (l>N),\qquad e\ge-2gC_{N,l}.
\end{equation*}
A Hall factor of charge $a$ in a nonzero summand satisfies
\begin{equation*}
 J-2g(N-a)^2\le j\le2ga^2.
\end{equation*}
To prove both statements, write the charges as $a_0\ge0$, $a_1,\ldots,a_l\ge1$, $m\ge0$ with sum $N$.
The factor bounds give $J\le2g\sum_ba_b^2+2m\le2g(\sum_ba_b^2+m^2)$.
The maximum square sum is $C_{N,l}$: place every remaining unit in one of the required positive slots.
Removing one Hall factor gives the second inequality.
Thus every fixed $(N,J)$ part is finite-dimensional.
Exactness of minimum-weight truncation gives the bar resolution over $\mathcal A^{\min}$ as well.

For the collision window $|J|\le\kappa_N$, every Hall factor needed therefore lies in
\[ -\kappa_N-2g(N-a)^2\le j\le2ga^2.\]
This larger interval is necessary: an operator killing the vacuum may act after another operator.
The annihilator is a left ideal,
not necessarily a two-sided ideal.
In computations using $\mathscr D_a$, its coefficient differentials must also be retained before taking minimum weights.
These finite bounds organize the calculation, but neither determine those differentials nor provide the generators and relations of the annihilator in all lengths.

\bibliographystyle{alpha}
\bibliography{reference1010}

\end{document}